%% file: houdayer-raum-maan-final.tex
\documentclass[leqno,11pt]{article}
\usepackage[bookmarks=false]{hyperref}

\usepackage[T1]{fontenc}
\usepackage[english]{babel}
\usepackage{lmodern}

\usepackage[intlimits]{amsmath}
\usepackage{amssymb, amsfonts}
\usepackage{amsthm}

\usepackage{MnSymbol}
\usepackage{mathbbol}
\usepackage{bbm}
\usepackage{mathrsfs}

\usepackage[all]{xy}

\usepackage{graphicx}
\usepackage{color}

\newcommand{\core}{\mathord{\text{\rm c}}}

\numberwithin{equation}{section}

\newtheorem{theoremcounter}{theoremcounter}[section]
\newtheorem{thmstarcounter}{thmstarcounter}

\newtheorem{corollary}[theoremcounter]{Corollary}
\newtheorem{lemma}[theoremcounter]{Lemma}

\newtheorem{proposition}[theoremcounter]{Proposition}
\newtheorem{theorem}[theoremcounter]{Theorem}
\newtheorem{thmstar}[thmstarcounter]{Theorem}

\theoremstyle{definition}

\newtheorem{definition}[theoremcounter]{Definition}

\newcommand{\dpr}{^{\prime\prime}}
\newcommand{\Ball}{\mathord{\text{\rm Ball}}}

\input{shortcuts.tex}

\begin{document}
\begin{center}
  \textbf{\LARGE Asymptotic structure of free Araki-Woods factors}
 
 \vspace{5mm}
  
  {\large by Cyril Houdayer\footnote{Research supported by the ANR Grant NEUMANN and JSPS Invitation Fellowship Program for Research in Japan FY2014} and Sven Raum\footnote{Research supported by the ANR Grant NEUMANN}}
\end{center}


\begin{abstract}
\noindent
The purpose of this paper is to investigate the structure of Shlyakhtenko's free Araki-Woods factors using the framework of ultraproduct von Neumann algebras. We first prove that all the free Araki-Woods factors $\Gamma(H_{\mathbb R}, U_t)^{\prime \prime}$ are $\omega$-solid in the following sense: for every von Neumann subalgebra $Q \subset \Gamma(H_{\mathbb R}, U_t)^{\prime \prime}$ that is the range of a faithful normal conditional expectation and such that the relative commutant $Q' \cap M^\omega$ is diffuse, we have that $Q$ is amenable. Next, we prove that the continuous cores of the free Araki-Woods factors $\Gamma(H_{\mathbb R}, U_t)^{\prime \prime}$ associated with mixing orthogonal representations $U : \mathbb R \to \mathcal O(H_{\mathbb R})$ are $\omega$-solid type ${\rm II_\infty}$ factors. Finally, when the orthogonal representation $U : \mathbb R \to \mathcal O(H_{\mathbb R})$ is weakly mixing, we prove a dichotomy result for all the von Neumann subalgebras $Q \subset \Gamma(H_{\mathbb R}, U_t)^{\prime \prime}$ that are globally invariant under the modular automorphism group $(\sigma_t^{\varphi_U})$ of the free quasi-free state $\varphi_U$.
\end{abstract}

\section{Introduction and statement of the main results}
\label{sec:introduction}

{\em Free Araki-Woods factors} were introduced by Shlyakhtenko in \cite{shlyakhtenko97}. In the context of Voiculescu's free probability theory, these factors can be regarded as the analogues of the hyperfinite factors coming from the canonical anticommutation relations (CAR) functor. Alternatively, they can also be regarded as the analogues of the free group factors in the setting of type ${\rm III}$ factors.

Following \cite{shlyakhtenko97}, to any orthogonal representation $U : \RR \to \mathcal O(H_\RR)$ on a separable real Hilbert space, one  associates a von Neumann algebra denoted by $\Gamma(H_\RR, U_t)\dpr$, called the {\it free Araki-Woods} von Neumann algebra. The von Neumann algebra $\Gamma(H_\RR, U_t)\dpr$ comes equipped with a unique {\it free quasi-free state} $\varphi_U$ that is always normal and faithful (see Subsection \ref{sec:FAW-factors} for a detailed construction). We have $\Gamma(H_\RR, \id)\dpr \cong \rL(\mathbb F_{\dim (H_\RR)})$ when $U : \RR \to \mathcal O(H_\RR)$ is the trivial representation and $\Gamma(H_\RR, U_t)\dpr$ is a full type ${\rm III}$ factor when $U : \RR \to \mathcal O(H_\RR)$ is not the trivial representation.

Free Araki-Woods factors were first studied using the framework of Voiculescu's free probability theory. A complete description of their type classification as well as fullness and computation of their Connes's $\tau$ and Sd invariants was obtained in \cite{shlyakhtenko97, shlyakhtenko98-applications, shlyakhtenko99} (see also the survey \cite{vaes04-etas-quasi-libres-libres}). More recently, free Araki-Woods factors were studied using the framework of Popa's Deformation/Rigidity theory \cite{popa07-deformation-rigidity}. This new approach allowed to obtain various indecomposability results in \cite{houdayer10} and complete metric approximation property and absence of Cartan subalgebra in \cite{houdayerricard11-araki-woods}. Because of their rich structure, free Araki-Woods factors form one of the most prominent classes of type ${\rm III}$ factors.

The purpose of this paper is to investigate the {\em asymptotic structure} of free Araki-Woods factors using the framework of ultraproduct von Neumann algebras. Before stating our main results, we first introduce some terminology.

We will say that a von Neumann subalgebra $Q \subset M$ is {\em with expectation} if there exists a faithful normal conditional expectation $\rE_Q : M \to Q$. We will say that a diffuse von Neumann algebra $M$ is {\em solid} if for every von Neumann subalgebra $Q \subset M$ with expectation whose relative commutant $Q' \cap M$ is diffuse, we have that $Q$ is amenable \cite{ozawa04-solid}. The first class of solid von Neumann algebras was discovered by Ozawa in \cite{ozawa04-solid}. He showed that every Gromov-word hyperbolic group $G$ gives rise to a solid von Neumann algebra $\rL(G)$. More conceptually, Ozawa showed that every {\em finite} diffuse von Neumann algebra satisfying the Akemann-Ostrand property (abbreviated property (AO) hereafter, see Subsection \ref{sec:AO}) is solid. It was observed in \cite{vaesvergnioux05} that in fact every diffuse von Neumann algebra satisfying property (AO) is solid.

Let now $\omega \in \beta(\NN) \setminus \NN$ be a non-principal ultrafilter. We refer to Subsection \ref{sec:ultraproducts} for the construction of the ultraproduct von Neumann algebra $M^\omega$. We will say that a von Neumann algebra $M$ is $\omega$-{\em solid} if for every von Neumann subalgebra $Q \subset M$ with expectation whose relative commutant $Q' \cap M^\omega$ is diffuse, we have that $Q$ is amenable. Since $M$ sits in $M^\omega$ as a von Neumann subalgebra with expectation, any $\omega$-solid von Neumann is obviously solid. As of today, the converse implication is an open problem\footnote{The proof of \cite[Proposition 7]{ozawa04-solid} requires $\mathcal N_0 = \mathcal M$ and only shows that any finite diffuse von Neumann algebra that is solid and that has property Gamma is amenable.}. Ozawa proved  in \cite{ozawa10-comment} that any finite diffuse von Neumann algebra satisfying property (AO) is $\omega$-solid. Our first result generalises Ozawa's result \cite{ozawa10-comment} to {\em arbitrary} diffuse von Neumann algebras with separable predual satisfying property (AO).

\begin{thmstar}
  \label{thm:omega-solidity}
  Any von Neumann algebra with separable predual satisfying property (AO) is $\omega$-solid.  In particular, any free Araki-Woods factor $\Gamma(H_\RR, U_t)\dpr$ associated with an orthogonal representation $U: \RR \to \mathcal O(H_\RR)$ on a separable real Hilbert space is $\omega$-solid.
\end{thmstar}

The proof of Theorem~\ref{thm:omega-solidity} combines Ozawa's original argument \cite{ozawa04-solid} together with several techniques from \cite{andohaagerup12} on the structure of ultraproduct von Neumann algebras. The proof of Theorem~\ref{thm:omega-solidity} is carried out in Section \ref{sec:omega-solidity}. The fact that all free Araki-Woods factors satisfy property~(AO) was proven in \cite[Chapter 4]{houdayer07}. It follows from Theorem \ref{thm:omega-solidity} that any von Neumann subalgebra with expectation and with property Gamma of any free Araki-Woods factor is necessarily amenable. We also show in Proposition \ref{prop:omega-solid-characterisation} that for every $\omega$-solid von Neumann algebra $M$ and every von Neumann subalgebra $Q \subset M$ with expectation and with no amenable direct summand, the relative commutant $Q' \cap M^\omega$ is necessarily discrete and hence equal to $Q' \cap M$ (see Theorem~\ref{thm:diffuse-centraliser}).

An interesting motivation for studying $\omega$-solidity in the setting of type ${\rm III}$ factors is the fact that Connes's $\tau$-invariant \cite{connes74-almost-periodic} is computable for all the $\omega$-solid type ${\rm III_1}$ factors that possess faithful normal states with non-amenable centralizer. More precisely, we show in Proposition \ref{prop:calculation-tau} that for every $\omega$-solid factor $M$ and for every faithful normal state $\varphi \in M_\ast$ such that the centralizer $M^\varphi$ is a non-amenable ${\rm II_1}$ factor, Connes's invariant $\tau(M)$ is the weakest topology on $\RR$ that makes the map $\RR \to \Aut(M) : t \mapsto \sigma_t^\varphi$ continuous.

Extending \cite[Theorem 1.2]{houdayer10}, we next show that the continuous cores of the free Araki-Woods factors associated with mixing orthogonal representations $U : \RR \to \mathcal O(H_\RR)$ are $\omega$-solid and so are their finite corners.  

\begin{thmstar}
  \label{thm:core-omega-solid}
  Let $U: \RR \ra \cO(H_\RR)$ be any orthogonal representation on a separable real Hilbert space that is the direct sum of a mixing representation and a representation of dimension less than or equal to $1$.  Let $M = \Gamma(H_\RR, U_t)\dpr$ be the associated free Araki-Woods factor.  Then its continuous core $ \core(M)$ is an $\omega$-solid type ${\rm II_\infty}$ factor.
\end{thmstar}

For the proof of Theorem \ref{thm:core-omega-solid}, we can no longer rely on property (AO). Instead, we work within the framework of Popa's Deformation/Rigidity theory \cite{popa07-deformation-rigidity} and we apply Popa's spectral gap rigidity \cite{popa08-spectral-gap} to the free malleable deformation of the free Araki-Woods factors arising from second quantisation (see Subsection \ref{sec:FAW-factors} for details). The proof of Theorem \ref{thm:core-omega-solid} is carried out in Section \ref{sec:omega-solidity}.

When dealing with {\em weakly mixing} orthogonal representations $U : \RR \to \mathcal O(H_\RR)$, we obtain a dichotomy result for all von Neumann subalgebras of the free Araki-Woods factors $\Gamma(H_\RR, U_t)\dpr$ that are globally invariant under the modular automorphism group of the free quasi-free state. This result constitutes a new feature in the structure theory of type ${\rm III}$ factors.

\begin{thmstar}
  \label{thm:dichotomy-FAW}
  Let $U : \RR \ra \cO(H_\RR)$ be any weakly mixing orthogonal representation on a separable real Hilbert space and $(M, \varphi) = (\Gamma(H_\RR, U_t)\dpr, \varphi_U)$ the associated free Araki-Woods factor. Let $Q \subset M$ be any von Neumann subalgebra that is globally invariant under the modular automorphism group $(\sigma_t^{\varphi})$ of the free quasi-free state $\varphi$. Then either $Q = \CC 1$ or $Q$ is a full non-amenable type ${\rm III_1}$ factor such that $Q' \cap M^\omega = \CC 1$.
\end{thmstar}

Theorem \ref{thm:dichotomy-FAW} shows in particular that any amenable von Neumann subalgebra of $M$ that is globally invariant under the modular automorphism group $(\sigma_t^\varphi)$ is necessarily trivial. Note that if the orthogonal representation $U : \RR \to \mathcal O(H_\RR)$ is not weakly mixing then the centralizer $M^\varphi$ is not trivial. This shows that the assumption of $U : \RR \to \mathcal O(H_\RR)$ being weakly mixing is necessary in Theorem \ref{thm:dichotomy-FAW}. The proof of Theorem \ref{thm:dichotomy-FAW} is based on the recent work of the first named author \cite{houdayer12, houdayer12-structure, houdayer14-gamma-stability} and uses in a novel fashion Popa's {\em asymptotic orthogonality property} \cite{popa83-maximal-injective-factors} in the framework of ultraproduct von Neumann algebras. The proof of Theorem \ref{thm:dichotomy-FAW} is carried out in Section \ref{sec:dichotomy}.

\subsection*{Acknowledgments}

This paper was completed when the first named author was visiting the Research Institute for Mathematical Sciences (RIMS) in Kyoto during Summer 2014. He warmly thanks Narutaka Ozawa and the RIMS for their kind hospitality. The authors also thank Stefaan Vaes for useful remarks regarding a first draft of this manuscript. Finally, the authors thank the anonymous referees for carefully reading the paper and providing valuable comments.

\section{Preliminaries}
\label{sec:preliminaries}

For a von Neumann algebra $M$, we will denote by $\mathcal Z(M)$ the centre of $M$, by $\mathcal U(M)$ the group of unitaries in $M$ and by $\Ball(M)$ the unit ball of $M$ with respect to the uniform norm $\|\cdot\|_\infty$. 

Let now $M$ be any $\sigma$-finite von Neumann algebra and $\varphi \in M_\ast$ any faithful normal state. We denote by $\rL^2(M, \varphi)$ (or simply $\rL^2(M)$ when no confusion is possible) the GNS $\rL^2$-completion of $M$ with respect to the inner product defined by $\langle x, y \rangle_\varphi = \varphi(y^*x)$ for all $x, y \in M$. We  denote by $\Lambda_\varphi : M \to \rL^2(M) : x \mapsto \Lambda_\varphi(x)$ the canonical embedding and by $J_\varphi : \rL^2(M) \to \rL^2(M)$ the canonical conjugation. We have $x \Lambda_\varphi (y) = \Lambda_\varphi(xy)$ for all $x, y \in M$. 

We will write $\|x\|_\varphi = \varphi(x^* x)^{1/2}$ and $\|x\|_\varphi^\# = \vphi(x^*x + xx^*)^{1/2}$ for all $x \in M$. Recall that on $\Ball(M)$, the topology given by $\|\cdot\|_\varphi$ (resp.\ $\|\cdot\|_\varphi^\#$) coincides with the strong  (resp.\ $\ast$-strong) topology. When $\varphi = \tau$ is a faithful normal tracial state, we will simply write $\|x\|_2 = \tau(x^*x)^{1/2}$ for all $x \in M$. We will say that a von Neumann algebra $M$ is {\em tracial} if it is endowed with a faithful normal tracial state $\tau$. 

\subsection{The continuous core of a $\sigma$-finite von Neumann algebra}
\label{sec:continuous core}

Let $(M, \vphi)$ be any $\sigma$-finite von Neumann algebra endowed with a faithful normal state.  We denote by $(\sigma_t^\vphi)$ the modular automorphism group with respect to the state  $\vphi$.  The centraliser $M^\vphi$ of the state $\vphi$ is by definition the fixed point algebra of $(M, (\sigma_t^\vphi))$.  The {\em continuous core} of $M$ with respect to $\varphi$, denoted by $\core_\varphi(M)$, is the crossed product von Neumann algebra $M \rtimes_{\sigma^\vphi} \RR$.  The natural inclusion $\pi_\vphi:M \ra \core_\vphi(M)$ and the unitary representation $\lambda_\vphi: \RR \ra \core_\vphi(M)$ satisfy the {\em covariance} relation
\begin{equation*}
  \lambda_\vphi(s) \pi_\vphi(x) \lambda_\vphi(s)^*
  =
  \pi_\vphi(\sigma^\vphi_s(x))
  \quad
  \text{ for all }
  x \in M \text{ and all } s \in \RR
  \eqstop
\end{equation*}
There is a unique faithful normal conditional expectation $\rE_\vphi: \core_{\varphi}(M) \ra \rL_\varphi(\RR)$ satisfying $\rE_\vphi(x \lambda_\vphi(s)) = \vphi(x) \lambda_\vphi(s)$.  The semifinite faithful normal trace $f \mapsto \int_{\RR} \exp(-t)f(t)$ on $\Linfty(\RR)$ gives rise to a semifinite faithful normal trace $\Tr_\vphi$ on $\rL_\varphi(\RR)$ via the Fourier transform.  The formula $\Tr_\vphi = \Tr_\vphi \circ \rE_\vphi$ extends it to a semifinite faithful normal trace on $\core_\vphi(M)$.

Because of Connes's Radon-Nikodym cocycle theorem \cite[Th\'eor\`eme 1.2.1]{connes73-type-III} (see also \cite[Theorem VIII.3.3]{takesaki03-III}), the semifinite von Neumann algebra $\core_\varphi(M)$ together with its trace $\Tr_\vphi$ does not depend on the choice of $\vphi$ in the following precise sense. If $\psi$ is another faithful normal state on $M$, there is a canonical surjective $*$-isomorphism
$\Pi_{\psi,\vphi} : \core_\varphi(M) \to \core_{\psi}(M)$ such that $\Pi_{\psi,\vphi} \circ \pi_\vphi = \pi_\psi$ and $\Tr_\psi \circ \Pi_{\psi,\vphi} = \Tr_\vphi$. Note however that $\Pi_{\psi,\vphi}$ does not map the subalgebra $\rL_\varphi(\RR) \subset \core_\varphi(M)$ onto the subalgebra $\rL_\psi(\RR) \subset \core_\psi(M)$.

\subsection{Free Araki-Woods factors}
\label{sec:FAW-factors}

Let $U : \RR \ra \cO(H_\RR)$ be any orthogonal representation on a separable real Hilbert space. 
Denote by $H = H_\RR \otimes_\RR \CC$ the complexified Hilbert space of $H_\RR$ and by $U : \RR \to \mathcal U(H)$ the corresponding unitary representation. Let $A$ be the positive selfadjoint closed operator defined on $H$ satisfying $A^{{\rm i} t} = U_t$ for all $t \in \RR$.  Then there is an isometric embedding of $H_\RR$ into $H$ given by 
\begin{equation*}
H_\RR \to H :  \xi \mapsto \bigl (\frac{2}{1 + A^{-1}} \bigr )^{1/2} \xi
  \text{,}
\end{equation*}
whose image we denote by $K_\RR$.  One can check that $K_\RR \cap {\rm i} K_\RR = \{0\}$ and $K_\RR + {\rm i} K_\RR$ is dense in $H$.  We denote by $J$ the canonical conjugation on $H = H_\RR \oplus {\rm i} H_\RR$ and by $I = J A^{-1/2}$. Then $I$ is an invertible anti-linear closed operator on $H$ satisfying $I = I^{-1}$. Oserve that $K_\RR = \{\xi \in {\rm dom}(T) : I \xi = \xi\}$. From now on, we will simply write $I : \xi + {\rm i} \eta \mapsto \overline{\xi + {\rm i} \eta} = \xi - {\rm i} \eta$ for all $\xi, \eta \in K_\RR$.

The {\em full Fock space} of $H$ is given by
\begin{equation*}
  \mathcal{F}(H) = \CC \Omega \oplus \bigoplus_{n = 1}^\infty H^{\otimes n}
  \text{.}
\end{equation*}
We call the vector $\Omega \in \mathcal F(H)$ the {\em vacuum vector}.  For all $\xi \in H$, the {\em left creation operator} $\ell(\xi) \in \mathcal B(\mathcal{F}(H))$ is given by the formulae
\begin{equation*}
  \ell(\xi)\Omega = \xi
  \quad \text{ and } \quad
  \ell(\xi)(\xi_1 \otimes \dotsm \otimes \xi_n) = \xi \otimes \xi_1 \otimes \dotsm \otimes \xi_n
  \text{.}
\end{equation*}
Note that $\|\ell(\xi)\| = \|\xi\|$ and $\ell(\xi)$ is an isometry if $\|\xi\| = 1$. Put $W(\xi) = \ell(\xi) + \ell(\xi)^*$ for all $\xi \in K_\RR$.  Following \cite{shlyakhtenko97}, we define the {\em free Araki-Woods factor} associated with $U : \RR \to \mathcal O(H_\RR)$ by
\begin{equation*}
  \Gamma(H_\RR, U_t)\dpr
  =
  \{ W(\xi) \, | \, \xi \in K_\RR\}\dpr
  \text{.}
\end{equation*}

The vector state $\varphi_U(x) = \langle x \Omega, \Omega \rangle$ on $\Gamma(H_\RR, U_t)\dpr$ is called the {\em free quasi-free state}.  It is faithful and one can show that the modular automorphism group of $\vphi_U$ is given by $\sigma_t^{\varphi_U} = \Ad(\mathcal{F}(U_t))$ for all $t \in \RR$, where $\mathcal{F}(U_t) = 1 \oplus \bigoplus_{n \geq 1} U_t^{\otimes n}$.  In particular, we have $\sigma_t^{\varphi_U}(W(\xi)) = W(U_t \xi)$ for all $\xi \in K_\RR$.

The GNS-representation of $\Gamma(H_\RR, U_t)\dpr$ with respect to $\varphi_U$ is isomorphic with its representation on $\mathcal{F}(H)$ with cyclic vector $\Omega$.  It is easy to check that for all $n \geq 1$ and all $\xi_1, \dotsc, \xi_n \in K_\RR + {\rm i} K_\RR$ there is a unique element $W(\xi_1 \otimes \dotsm \otimes \xi_n) \in \Gamma(H_\RR, U_t)\dpr$ such that $W(\xi_1 \otimes \dotsm \otimes \xi_n)\Omega = \xi_1 \otimes \dotsm \otimes \xi_n$. We have $W(\xi) = \ell(\xi) + \ell(\overline \xi)^*$ for all $\xi \in K_\RR + {\rm i} K_\RR$. The following proposition describes a Wick-type formula for such elements.

\begin{proposition}[\cite{houdayer12, houdayer12-structure, houdayerricard11-araki-woods}]
  \label{prop:wick}
  Let $\xi_j , \eta_k \in K_\RR + {\rm i} K_\RR$, for $j, k \geq 1$. The following statements are true:
  \begin{enumerate}
  \item The Wick formula
    $ \displaystyle
    W(\xi_1 \otimes \cdots \otimes \xi_n)
    =
    \sum_{k = 0}^n \ell(\xi_1) \cdots \ell(\xi_k) \ell(\overline \xi_{k + 1})^* \cdots \ell(\overline \xi_n)^*$ holds.    
  \item The product $W(\xi_1 \otimes \cdots \otimes \xi_r) W(\eta_1 \otimes \cdots \otimes \eta_s)$ equals
    \begin{equation*}
      W(\xi_1 \otimes \cdots \otimes \xi_r \otimes \eta_1 \otimes \cdots \otimes \eta_s) + 
      \langle \overline \xi_r, \eta_1\rangle W(\xi_1 \otimes \cdots \otimes \xi_{r - 1}) W(\eta_2 \otimes \cdots \otimes \eta_s)
      \eqstop
    \end{equation*}
  \item We have $W(\xi_1 \otimes \cdots \otimes \xi_n)^* = W(\overline \xi_n \otimes \cdots \otimes \overline \xi_1)$.
  \item The linear span of $\{1, W(\xi_1 \otimes \cdots \otimes \xi_n) : n \geq 1, \xi_i \in K_\RR + {\rm i} K_\RR \}$ forms a unital $\sigma$-strongly dense $\ast$-subalgebra of $\Gamma(H_\RR, U_t)\dpr$.
  \end{enumerate}
\end{proposition}

\begin{proof}
  The proof of (i) is borrowed from \cite[Lemma 3.2]{houdayerricard11-araki-woods}.  We prove the formula by induction on $n$. For $n \in \{0, 1\}$, we have $W(\Omega)=1$ and we already observed that $W(\xi_i)=\ell(\xi_i) + \ell(\overline \xi_i)^*$.

  Next, for $\xi_{0}\in K_\RR + {\rm i} K_\RR$, we have 
  \begin{align*}
    W(\xi_0)W(\xi_1\otimes \cdots \otimes \xi_n)\Omega & = W(\xi_0)(\xi_1\otimes \cdots \otimes \xi_n)
    \\ 
    & =
    (\ell(\xi_0)+\ell(\overline \xi_0)^*)  \xi_1\otimes \cdots \otimes \xi_n \\ 
    & =
    \xi_0\otimes \xi_1\otimes \cdots \otimes \xi_n + \langle \overline \xi_0,\xi_{1}\rangle \, \xi_2\otimes \cdots \otimes \xi_n.
  \end{align*}
  So, we obtain 
  \begin{align*}
    W(\xi_0\otimes \cdots \otimes \xi_n) & = W(\xi_0)W(\xi_1\otimes \cdots \otimes \xi_n) - \langle \overline \xi_0,\xi_1\rangle W(\xi_2\otimes \cdots \otimes \xi_n) \\
    & =
    \ell(\overline \xi_0)^*W(\xi_1\otimes \cdots \otimes \xi_n) - \langle \overline \xi_0,\xi_1 \rangle W(\xi_2\otimes \cdots \otimes \xi_n)
    + \ell(\xi_0) W(\xi_1 \otimes \cdots \otimes \xi_n).
  \end{align*}
  Using the assumption for $n$ and $n-1$ and the relation $\ell(\overline \xi_0)^*\ell(\xi_1) = \langle \overline \xi_0, \xi_1\rangle$, we obtain
  \begin{equation*}
    \ell(\overline \xi_0)^*W(\xi_1\otimes \cdots \otimes \xi_n)
    =
    \langle \overline \xi_0, \xi_1 \rangle W(\xi_2\otimes \cdots \otimes \xi_n)
    + \ell(\overline \xi_0)^*\ell(\overline \xi_1)^* \cdots \ell(\overline \xi_n)^*
  \eqstop
  \end{equation*}
  Since $\ell(\xi_0)W(\xi_1\otimes \cdots \otimes \xi_n)$ gives the last $n + 1$ terms in the Wick formula at order $n+1$ and $\ell(\overline \xi_0)^*\ell(\overline \xi_1)^* \cdots \ell(\overline \xi_n)^*$ gives the first term, we are done.

  We now prove (ii).  By the Wick formula, we have that $W(\xi_1 \otimes \cdots \otimes \xi_r) W(\eta_1 \otimes \cdots \otimes \eta_s)$ is equal to 
  \begin{equation*}
    \sum_{0 \leq j \leq r, 0 \leq k \leq s}
    \ell(\xi_1) \cdots \ell(\xi_j) \ell(\overline \xi_{j + 1})^* \cdots \ell(\overline \xi_r)^* \ell(\eta_1) \cdots \ell(\eta_k) \ell(\overline \eta_{k + 1})^* \cdots \ell(\overline \eta_s)^*
    \eqstop
  \end{equation*}
  Recall that we have $\ell(\overline \xi_r)^*\ell(\eta_1) = \langle \overline \xi_r, \eta_1 \rangle$. Therefore the above sum equals
  \begin{align*}
    & \Bigl ( \sum_{0 \leq j \leq r - 1} \ell(\xi_1) \cdots \ell(\xi_j) \ell(\overline \xi_{j + 1})^* \cdots \ell(\overline \xi_r)^* \ell(\overline \eta_1)^* \cdots \ell(\overline \eta_s)^* \\
    & \qquad
    + \sum_{0 \leq k \leq s} \ell(\xi_1) \cdots \ell(\xi_r) \ell(\eta_1) \cdots \ell(\eta_k) \ell(\overline \eta_{k + 1})^* \cdots \ell(\overline \eta_s)^* \Bigr) \\
    & \qquad
    + \langle \overline \xi_r, \eta_1\rangle \sum_{0 \leq j \leq r - 1,  1 \leq k \leq s} 
    \ell(\xi_1) \cdots \ell(\xi_j) \ell(\overline \xi_{j + 1})^* \cdots \ell(\overline \xi_{r - 1})^* \ell(\eta_2) \cdots \ell(\eta_k) \ell(\overline \eta_{k + 1})^* \cdots \ell(\overline \eta_s)^*
    \eqstop
  \end{align*}
  Therefore $W(\xi_1 \otimes \cdots \otimes \xi_r) W(\eta_1 \otimes \cdots \otimes \eta_s)$ is equal to 
  \begin{equation*}
      W(\xi_1 \otimes \cdots \otimes \xi_r \otimes \eta_1 \otimes \cdots \otimes \eta_s) + 
  \langle \overline \xi_r, \eta_1\rangle W(\xi_1 \otimes \cdots \otimes \xi_{r - 1}) W(\eta_2 \otimes \cdots \otimes \eta_s).
  \end{equation*}
  It is now clear that (i) $\Rightarrow$ (iii).  Moreover, (iv) follows from (iii) using an induction procedure.
\end{proof}

\subsection{Ultraproduct von Neumann algebras}
\label{sec:ultraproducts}

Fix a non-principal ultrafilter $\omega \in \beta(\NN) \setminus \NN$.  Let $M$ be any $\sigma$-finite von Neumann algebra. Define
\begin{align*}
\mathcal I_\omega(M) &= \left\{ (x_n)_n \in \ell^\infty(\NN, M) \amid x_n \to 0 \ast\text{-strongly as } n \to \omega \right\} \\
\mathcal M^\omega(M) &= \left \{ (x_n)_n \in \ell^\infty(\NN, M) \amid  (x_n)_n \, \mathcal I_\omega(M) \subset \mathcal I_\omega(M) \text{ and } \mathcal I_\omega(M) \, (x_n)_n \subset \mathcal I_\omega(M)\right\}.
\end{align*}

The {\em multiplier algebra} $\mathcal M^\omega(M)$ is a C$^*$-algebra and $\mathcal I_\omega(M) \subset \mathcal M^\omega(M)$ is a norm closed two-sided ideal. Following \cite[Chapter 5]{ocneanu85}, we define the {\em ultraproduct von Neumann algebra} $M^\omega$ by $M^\omega = \mathcal M^\omega(M) / \mathcal I_\omega(M)$. We denote the image of $(x_n)_n \in \mathcal M^\omega(M)$ by $(x_n)^\omega \in M^\omega$. 

For all $x \in M$, the constant sequence $(x)_n$ lies in the multiplier algebra $\mathcal M^\omega(M)$. We will identify $M$ with $(M + \mathcal I_\omega(M))/ \mathcal I_\omega(M)$ and regard $M \subset M^\omega$ as a von Neumann subalgebra. The map $\rE_\omega : M^\omega \to M : (x_n)^\omega \mapsto \sigma \text{-weak}\lim_{n \to \omega} x_n$ is a faithful normal conditional expectation. For every faithful normal state $\varphi \in M_\ast$, the formula $\varphi^\omega = \varphi \circ \rE_\omega$ defines a faithful normal state on $M^\omega$. Observe that $\varphi^\omega((x_n)^\omega) = \lim_{n \to \omega} \varphi(x_n)$ for all $(x_n)^\omega \in M^\omega$. 

Put $\mathcal H = \rL^2(M, \varphi)$. The {\em ultraproduct Hilbert space} $\mathcal H^\omega$ is defined to  be the quotient of $\ell^\infty(\NN, \mathcal H)$ by the subspace consisting in sequences $(\xi_n)_n$ satisfying $\lim_{n \to \omega} \|\xi_n\|_{\mathcal H} = 0$. We denote the image of $(\xi_n)_n \in \ell^\infty(\NN, \mathcal H)$ by $(\xi_n)_\omega \in \mathcal H^\omega$. The inner product space structure on the Hilbert space $\mathcal H^\omega$ is defined by $\langle (\xi_n)_\omega, (\eta_n)_\omega\rangle_{\mathcal H^\omega} = \lim_{n \to \omega} \langle \xi_n, \eta_n\rangle_{\mathcal H}$. The GNS Hilbert space $\rL^2(M^\omega, \varphi^\omega)$ can be embedded into $\mathcal H^\omega$ as a closed subspace by $\Lambda_{\varphi^\omega}((x_n)^\omega) \mapsto (\Lambda_\varphi(x_n))_\omega$.

Put $x \varphi = \varphi (\cdot \,  x)$ and $\varphi x = \varphi(x \, \cdot)$ for all $x \in M$ and all $\varphi \in M_\ast$. We will be using the following well-known proposition.

\begin{proposition}
\label{prop:representation-by-projections}
Let $(M, \varphi)$ be any $\sigma$-finite von Neumann algebra endowed with a faithful normal state.
\begin{enumerate}
\item For every $\lambda > 0$ and every $(x_n)_n \in \ell^\infty(\NN, M)$ satisfying $\lim_{n \to \omega} \|x_n \varphi - \lambda \varphi x_n\| = 0$, we have $(x_n)_n \in \mathcal M^\omega(M)$ and $(x_n)^\omega \varphi^\omega = \lambda \varphi^\omega (x_n)^\omega$.
\item For every projection $e \in M^\omega$, there exists a sequence of projections $(e_n)_n \in \mathcal M^\omega(M)$ such that $e = (e_n)^\omega$. If $M$ is moreover diffuse, the projections $e_n \in M$ may be chosen such that $\varphi(e_n) = \varphi^\omega(e)$ for all $n \in \NN$.
\end{enumerate}
\end{proposition}

\begin{proof}
(i) Let $(x_n)_n \in \ell^\infty(\NN, M)$ such that $\lim_{n \to \omega} \|x_n \varphi - \lambda \varphi x_n\| = 0$. Let $(b_n)_n \in \mathcal I_\omega(M)$. We may assume that $\max \{\|x_n\|_\infty, \|b_n\|_\infty : n \in \NN\} \leq 1$. Using the Cauchy-Schwarz inequality, for all $n \in \NN$, we have
\begin{align*}
(\|x_n b_n\|_\varphi^\#)^2 &= \varphi(b_n^* \, x_n^* x_n b_n) + \varphi(x_n \, b_n b_n^* x_n^*) \\
& \leq \|b_n\|_\varphi \, \|x_n^* x_n b_n\|_\varphi + \lambda^{-1} |(x_n \varphi - \lambda \varphi x_n)(b_n b_n^* x_n^*)| + \lambda^{-1} |\varphi(b_n \, b_n^* x_n^* x_n)| \\
& \leq \|b_n\|_\varphi + \lambda^{-1}\|x_n \varphi - \lambda\varphi x_n\| \, \|b_n b_n^* x_n^*\|_\infty + \lambda^{-1} \|b_n^*\|_\varphi \, \|b_n^* x_n^* x_n\|_\varphi \\
& \leq \|b_n\|_\varphi + \lambda^{-1}\|x_n \varphi - \lambda\varphi x_n\| + \lambda^{-1} \|b_n^*\|_\varphi.
\end{align*}
Therefore, we obtain $\lim_{n \to \omega} \|x_n b_n\|_\varphi^\# = 0$ and so $(x_n b_n)_n \in \mathcal I_\omega(M)$. Likewise, for all $n \in \NN$, we have
\begin{align*}
(\|b_n x_n\|_\varphi^\#)^2 &= \varphi(x_n^* \, b_n^* b_n x_n) + \varphi(b_n \, x_n x_n^* b_n^*) \\
& \leq  |(\lambda x_n^* \varphi - \varphi x_n^*)(b_n^* b_n x_n)| + \lambda |\varphi(b_n^* \, b_n x_n x_n^*)| + \|b_n^*\|_\varphi \, \|x_n x_n^* b_n^*\|_\varphi \\
& \leq  \|\lambda x_n^* \varphi - \varphi x_n^*\| \, \|b_n^* b_n x_n\|_\infty+ \lambda\|b_n\|_\varphi \, \|b_n x_n x_n^*\|_\varphi + \|b_n^*\|_\varphi \\
& \leq  \|x_n \varphi - \lambda \varphi x_n\| + \lambda \|b_n\|_\varphi + \|b_n^*\|_\varphi.
\end{align*}
Therefore, we obtain $\lim_{n \to \omega} \|b_n x_n\|_\varphi^\# = 0$ and so $(b_n x_n)_n \in \mathcal I_\omega(M)$. This shows that $(x_n)_n \in \mathcal M^\omega(M)$. Moreover, $(x_n)^\omega \varphi^\omega = \lambda \varphi^\omega (x_n)^\omega$ by \cite[Lemma 4.36]{andohaagerup12}.

For the first part of the proof of (ii), see the proof of \cite[Proposition 2.4 (3)]{houdayer14-gamma-stability}. It remains to prove the moreover part of (ii) when $M$ is diffuse. Let $p \in M^\omega$ be any projection and $(p_n)_n \in \mathcal M^\omega(M)$ a sequence of projections such that $p = (p_n)^\omega$. Let $n \geq 1$. Assume that $\vphi(p_n) \geq \vphi^\omega(p)$. Since $p_n M p_n$ is diffuse, we may choose a projection $r_n \in p_n M p_n$ such that $\vphi(r_n) = \vphi^\omega(p)$. Assume that $\varphi(p_n) \leq \varphi^\omega(p)$. Since $(1 - p_n) M (1 - p_n)$ is diffuse, we may choose a projection $s_n \in (1 - p_n) M (1 - p_n)$ such that $\vphi(s_n) = \vphi^\omega(p) - \varphi(p_n)$. Put $r_n = p_n + s_n$. 

We obtain $\lim_{n \to \omega} \|p_n - r_n\|_\varphi^2 = \lim_{n \to \omega} |\varphi(p_n - r_n)| = 0$ and hence $(p_n - r_n)_n \in \mathcal I_\omega(M)$. Thus, we have $p = (r_n)^\omega$ and $\varphi(r_n) = \varphi^\omega(p)$ for all $n \in \NN$.
\end{proof}

The next theorem will be very useful to prove Theorem \ref{thm:omega-solidity}. It is a generalization of \cite[Lemma~2.7]{ioana12} to arbitrary von Neumann algebras.

\begin{theorem}
  \label{thm:diffuse-centraliser}
  Let $Q \subset M$ be an inclusion of von Neumann algebras with faithful normal conditional expectation $\rE_Q : M \to Q$. Assume that $Q$ has separable predual. Denote by $z \in \cZ(Q' \cap M^\omega)$ the unique maximal central projection such that $(Q' \cap M^\omega)z$ is discrete. Then
\begin{itemize}
\item $z \in \cZ(Q' \cap M^\omega) \cap \mathcal Z(Q' \cap M)$,
\item  $(Q' \cap (M^\omega)^{\vphi^\omega})(1 - z)$ is diffuse for all faithful normal states $\varphi \in M_\ast$ such that $\varphi \circ \rE_Q = \varphi$ and
\item $(Q' \cap M^\omega)z = (Q' \cap M)z$.
\end{itemize} 
\end{theorem}

We start by proving the following two lemmas.

\begin{lemma}
  \label{lem:discrete-asymptotic-commutant}
  Let $Q \subset M$ be an inclusion of von Neumann algebras with faithful normal conditional expectation $\rE_Q : M \to Q$. Assume that $Q$ has separable predual. Let $\varphi \in M_\ast$ be any faithful normal state such that $\varphi \circ \rE_Q = \varphi$. Denote by $e \in \cZ(Q' \cap (M^\omega)^{\vphi^\omega})$ the unique maximal central projection such that $(Q' \cap (M^\omega)^{\vphi^\omega})e$ is discrete. Then
  \begin{itemize}
  \item $e \in \cZ(Q' \cap (M^\omega)^{\vphi^\omega}) \cap \mathcal Z(Q' \cap M^\varphi)$ and
  \item $(Q' \cap (M^\omega)^{\vphi^\omega})e = (Q' \cap M^\varphi)e$.
  \end{itemize}
\end{lemma}

\begin{proof}
 The proof is a generalisation of \cite[Lemma~2.7]{ioana12} (see also the proof of \cite[Proposition~2.5]{houdayer14-gamma-stability}). Put $\mathcal Q = Q' \cap (M^\omega)^{\varphi^\omega}$ and denote by $e \in \mathcal Z(\mathcal Q)$ the unique maximal central projection in $\mathcal Q$ such that $\mathcal Q e$ is discrete. Choose a sequence of projections $(e_n)_n \in \mathcal M^\omega(M)$ such that $e = (e_n)^\omega$. Let $a = \sigma\text{-weak} \lim_{n \to \omega} e_n \in \mathcal Z(Q' \cap M^\varphi)$.

Next, we construct by induction a sequence of projections $(f_m)_{m \geq 1}$ in $\mathcal Q$ such that 
\begin{equation}\label{equality-lambda}
\varphi^\omega(ef_i) = \varphi(a^2), \; \varphi^\omega(ef_i a) = \varphi(a^3) \; \text{ and } \;  \varphi^\omega(ef_i f_j) = \varphi^\omega(e f_i a), \forall 1 \leq i < j.
\end{equation}
Indeed, assume that $f_1, \dots, f_m \in \mathcal Q$ have been constructed. For every $1 \leq j \leq m$, choose a sequence of projections $(f_{j, n})_n \in \mathcal M^\omega(M)$ such that $f_j = (f_{j,n})^\omega$. Let $(x_i)_{i \in \NN}$ be a $\|\cdot\|_\varphi^\#$-dense sequence in $\Ball(Q)$. Since $e = (e_n)^\omega \in (M^\omega)^{\varphi^\omega}$, $\lim_{n \to \omega} \|e_n x_i - x_i e_n\|_\varphi^\# = 0$ for all $i \in \NN$ and $e_n \to a$ $\sigma$-weakly as $n \to \omega$, we can find an increasing sequence $(k_n)_n$ in $\NN$ such that for every $n \geq 1$,  we have
\begin{enumerate}
\item [(P1)] $\|e_{k_n} \varphi - \varphi e_{k_n}\| \leq \frac1n$, 
\item [(P2)] $\|e_{k_n} x_i - x_i e_{k_n}\|_\varphi^\# \leq \frac1n$ for all $1 \leq i \leq n$,
\item [(P3)] $|\varphi(e_n e_{k_n}) - \varphi(e_n a)| \leq \frac1n$,
\item [(P4)] $|\varphi(e_n e_{k_n} a) - \varphi(e_n a^2)| \leq \frac1n$ and
\item [(P5)] $|\varphi(e_n f_{j, n} e_{k_n}) - \varphi(e_n f_{j, n} a)| \leq \frac1n$ for all $1 \leq j \leq m$.
\end{enumerate}

Properties (P1) and (P2) together with Proposition \ref{prop:representation-by-projections} imply that $(e_{k_n})_n \in \mathcal M^\omega(M)$ and $f = (e_{k_n})^\omega \in Q' \cap (M^\omega)^{\varphi^\omega}$. Property (P3) implies that $\varphi^\omega(ef) = \varphi^\omega(e a)= \varphi(a^2)$, Property (P4) implies that $\varphi^\omega(efa) = \varphi^\omega(e a^2) = \varphi(a^3)$ and Property (P5) implies that $\varphi^\omega(ef_j f) = \varphi^\omega(e f_j a)$ for all $1 \leq j \leq m$. We can now put $f_{m + 1} = f$. This finishes the proof of the induction.

Define $p_m = f_m e$ which is a projection in $\mathcal Qe$. We have $\varphi^\omega(p_j) = \varphi(a^2)$ and $\varphi^\omega(p_j p_m) = \varphi(a^3)$ for all $1 \leq j < m$. Observe that since $\mathcal Q e$ is a discrete tracial von Neumann algebra and hence a countable direct sum of finite dimensional factors, $\Ball(\mathcal Q e)$ is $\|\cdot\|_{\varphi^\omega}$-compact. Thus, we may choose a subsequence $(p_{m_k})_{k \geq 1}$ that is $\|\cdot\|_{\varphi^\omega}$-convergent in $\Ball(\mathcal Qe)$. By the Cauchy-Schwarz inequality, for all $1 \leq j < k$, we have
\begin{equation*}
  |\varphi^\omega(p_{m_j} p_{m_k}) - \varphi^\omega(p_{m_j}) |
  =
  |\varphi^\omega(p_{m_j}(p_{m_k} - p_{m_j}))| \leq \|p_{m_j} - p_{m_k}\|_{\varphi^\omega}
  \eqstop
\end{equation*}
Taking the limit as $(j, k) \to \infty$ and using (\ref{equality-lambda}), we obtain $\varphi(a^2) = \varphi(a^3)$ and so $0 \leq a \leq 1$ is a projection in $Q' \cap M^\varphi$. Thus we have $\|e_n - a\|^2_\varphi = \varphi(e_n) + \varphi(a) - 2 \varphi(e_n a) \to 0$ as $n \to \omega$ and so $e = (e_n)^\omega = a \in \mathcal Z(Q' \cap (M^\omega)^{\varphi^\omega}) \cap \mathcal Z(Q' \cap M^\varphi)$.

It remains to prove that $(Q' \cap (M^\omega)^{\vphi^\omega})e = (Q' \cap M^\varphi)e$. Assume by contradiction that this is not the case and choose a nonzero projection $f \in (Q' \cap (M^\omega)^{\vphi^\omega})e$ such that $f \notin (Q' \cap M^\varphi)e$. Let $\varphi_e$ be the faithful normal state on $eMe$ defined by $\varphi_e = \frac{\varphi(e \cdot e)}{\varphi(e)}$. Recall that $e \in M^\varphi$. Denote by $\rE_{eMe} : (eMe)^\omega \to eMe$ the canonical faithful normal conditional expectation. Recall  that $\varphi_e \circ \rE_{eMe} = \varphi_e^\omega$. Since $f \notin (Q' \cap M^\varphi)e$, we have $\|f - \rE_{eMe}(f)\|_{\varphi_e^\omega} > 0$. Moreover, for all $y \in \Ball(eMe)$, we have
\begin{equation*}
  \|f - y\|_{\varphi_e^\omega} \geq \|f - \rE_{eMe}(f)\|_{\varphi_e^\omega} > 0
  \eqstop
\end{equation*}

Put $\varepsilon = \frac12 \|f - \rE_{eMe}(f)\|_{\varphi_e^\omega}$ and $f_1 = f \in \mathcal Qe$. Next, we construct by induction a sequence of projections $f_m \in \mathcal Qe$ such that $\|f_p - f_q\|_{\varphi_e^\omega} \geq \varepsilon$ for all $p, q \geq 1$ such that $p \neq q$. Assume that $f_1, \dots, f_{m} \in \mathcal Qe$ have been constructed. For every $1 \leq j \leq m$, choose a sequence of projections $(f_{j, n})_n \in \mathcal M^\omega(eMe)$ such that $f_j = (f_{j,n})^\omega$. Let $(x_i)_{i \in \NN}$ be a $\|\cdot\|_{\varphi_e}^\#$-dense sequence in $\Ball(Qe)$. Since $f = f_1 = (f_{1,n})^\omega \in ((eMe)^\omega)^{\varphi_e^\omega}$, $\lim_{k \to \omega} \|f_{1, k} x_i - x_i f_{1, k}\|_{\varphi_e}^\# = 0$ for all $i \in \NN$ and $\lim_{k \to \omega} \|f_{1, k} - f_{j, n}\|_{\varphi_e} =  \|f - f_{j, n}\|_{\varphi_e^\omega} \geq 2 \varepsilon$ for all $1 \leq j \leq m$ and all $n \in \NN$, we can find an increasing sequence $(k_n)_n$ in $\NN$ such that for every $n \geq 1$,  we have
\begin{enumerate}
\item [(P1)] $\|f_{1, k_n} \varphi_e - \varphi_e f_{1, k_n}\| \leq \frac1n$, 
\item [(P2)] $\|f_{1, k_n} x_i - x_i f_{1, k_n}\|_{\varphi_e}^\# \leq \frac1n$ for all $1 \leq i \leq n$ and
\item [(P3)] $\|f_{1, k_n} - f_{j, n}\|_{\varphi_e} \geq \varepsilon$ for all $1 \leq j \leq m$.
\end{enumerate}

By the same reasoning as before, Properties (P1) and (P2) together with Proposition \ref{prop:representation-by-projections} imply that $(f_{1, k_n})_n \in \mathcal M^\omega(eMe)$ and $(f_{1, k_n})^\omega \in (Qe)' \cap ((eMe)^\omega)^{\varphi_e^\omega}= \mathcal Qe$. Moreover, Property (P3) implies that $\| (f_{1, k_n})^\omega - f_j\|_{\varphi_e^\omega} \geq \varepsilon$ for all $1 \leq j \leq m$. We can now put $f_{m + 1} = (f_{1, k_n})^\omega$. This finishes the proof of the induction.

So, we have constructed a sequence of projections $f_m \in \mathcal Qe$ such that $\|f_p - f_q\|_{\varphi_e^\omega} \geq \varepsilon$ for all $p, q \geq 1$ such that $p \neq q$. This however contradicts the fact that $\Ball(\mathcal Qe)$ is $\|\cdot\|_{\varphi_e^\omega}$-compact and finishes the proof Lemma \ref{lem:discrete-asymptotic-commutant}. 
\end{proof}

\begin{lemma}
  \label{lem:diffuse-centraliser}
  Let $Q \subset M$ be an inclusion of von Neumann algebras with faithful normal conditional expectation $\rE_Q : M \to Q$. Assume that $Q$ has separable predual. Let $\varphi \in M_\ast$ be any faithful normal state such that $\varphi \circ \rE_Q = \varphi$. If $Q' \cap (M^\omega)^{\vphi^\omega} = \CC$ then $Q' \cap M^\omega = \CC$.
\end{lemma}

\begin{proof}
The proof is a straightforward generalisation of \cite[Theorem 5.2]{andohaagerup12} and so we will only sketch it. 

Assume that $Q' \cap (M^\omega)^{\vphi^\omega} = \CC$. Since $(Q' \cap M^\omega)^{\varphi^\omega} = Q' \cap (M^\omega)^{\vphi^\omega} = \CC$, \cite[Lemma 5.3]{andohaagerup12} shows that $Q' \cap M^\omega = \CC$ or $Q' \cap M^\omega$ is a type ${\rm III_1}$ factor. By contradiction, assume that $Q' \cap M^\omega$ is a type ${\rm III_1}$ factor. Choose $(a_i)_{i \in \NN}$ a $\|\cdot\|_{\varphi}^\#$-dense sequence in $\Ball(Q)$. Proceeding as in the proof of \cite[Theorem 5.2]{andohaagerup12}, for all $n \in \NN$ and all $i, j \in \{1, 2\}$, we find elements $f_{ij}^{(n)} \in M$ that satisfy the conditions of \cite[Theorem 5.2, Claim 1]{andohaagerup12} with respect to the sequence $(a_i)_{i \in \NN}$ in $\Ball(Q)$. As in \cite[Theorem 5.2, Claim 2]{andohaagerup12}, we obtain that $(f_{ij}^{(n)})_n \in \mathcal M^\omega(M)$ for all $i, j \in \{1, 2\}$. Finally, we obtain a projection $g_{11} \in Q' \cap (M^\omega)^{\varphi^\omega}$ such that $\varphi^\omega(g_{11}) \neq 0, 1$. This is a contradiction and finishes the proof of Lemma \ref{lem:diffuse-centraliser}.
\end{proof}

\begin{proof}[Proof of Theorem \ref{thm:diffuse-centraliser}]
Let $Q \subset M$ be an inclusion of von Neumann algebras and fix a faithful normal state $\vphi \in M_\ast$ such that $\varphi \circ \rE_Q = \varphi$.  By Lemma~\ref{lem:discrete-asymptotic-commutant}, the unique maximal central projection $e \in \mathcal Z(Q' \cap (M^\omega)^{\varphi^\omega})$ such that $(Q' \cap (M^\omega)^{\vphi^\omega})e$ is discrete satisfies $e \in \mathcal Z(Q' \cap (M^\omega)^{\varphi^\omega}) \cap \mathcal Z(Q' \cap M^\varphi)$ and $(Q' \cap (M^\omega)^{\vphi^\omega})e = (Q' \cap M^\vphi)e$. Observe that the projection $e$ may {\em a priori} depend on the state $\varphi$. However, we will prove that this is not the case and show that the projection $e$ satisfies the conclusion of Theorem \ref{thm:diffuse-centraliser}.

Since $(Q' \cap (M^\omega)^{\vphi^\omega})e$ is discrete, choose a family $(p_i)_{i \in I}$ of pairwise orthogonal minimal projections in  $(Q' \cap (M^\omega)^{\vphi^\omega})e$ such that $\sum_{i \in I} p_i = e$. We have $p_i \in (Q' \cap M^\varphi)e$ and $(Qp_i)' \cap ((p_i M p_i)^\omega)^{\vphi_{p_i}^\omega} = \CC p_i$ where $\varphi_{p_i} = \frac{\varphi(p_i \cdot p_i)}{\varphi(p_i)}$. Lemma~\ref{lem:diffuse-centraliser} applied to the inclusion $Qp_i \subset p_i M p_i$ implies that
  \begin{equation*}
    p_i (Q ' \cap  M^\omega) p_i
    =
    (Q p_i) ' \cap (p_i M p_i)^\omega
    =
    \CC p_i
  \end{equation*}
and hence $p_i$ is a minimal projection in $e(Q' \cap M^\omega)e$. Since $\sum_{i \in I} p_i = e$, we have that $e(Q' \cap M^\omega)e$ is discrete. Denote by $z(e)$ the central support of the projection $e$ in $Q' \cap M^\omega$. We obtain that $(Q' \cap M^\omega)z(e)$ is still discrete. Since $z(e) \in \mathcal Z(Q' \cap M^\omega)$ and $Q' \cap M^\omega$ is globally invariant under the modular automorphism group $(\sigma_t^{\varphi^\omega})$, we have $z(e) \in \mathcal Z(Q' \cap (M^\omega)^{\varphi^\omega})$ and $(Q' \cap (M^\omega)^{\varphi^\omega})z(e)$ is discrete. By Lemma~\ref{lem:discrete-asymptotic-commutant} and since $e \leq z(e)$, we obtain $e = z(e) \in \mathcal Z(Q' \cap M^\omega) \cap \mathcal Z(Q' \cap M)$ and $(Q' \cap M^\omega)e$ is discrete. Observe that $(Q' \cap M^\omega)(1 - e)$ is diffuse since $(Q' \cap (M^\omega)^{\varphi^\omega})(1 - e)$ is a diffuse subalgebra with expectation. Therefore, $e \in \mathcal Z(Q' \cap M^\omega) \cap \mathcal Z(Q' \cap M)$ is the unique projection such that $(Q' \cap M^\omega)e$ is discrete and $(Q' \cap M^\omega)(1 - e)$ is diffuse and hence $e$ does not depend on the choice of the faithful normal state $\varphi \in M_\ast$ satisfying $\varphi = \varphi \circ \rE_Q$. Thus, the above proof shows that $(Q' \cap (M^\omega)^{\psi^\omega})(1 - e)$ is diffuse for all faithful normal states $\psi \in M_\ast$ satisfying $\psi = \psi \circ \rE_Q$.

It remains to prove that $(Q' \cap M^\omega)e = (Q' \cap M)e$. Recall that there exists a family $(p_i)_{i \in I}$ of pairwise orthogonal minimal projections in  $(Q' \cap M^\omega) e$ such that $\sum_{i \in I} p_i = e$ and $p_i \in (Q' \cap M^\varphi)e$ for all $i \in I$. In order to show that $(Q' \cap M^\omega)e = (Q' \cap M)e$, it suffices to prove that $e_{\mathcal F}(Q' \cap M^\omega)e_{\mathcal F} = e_{\mathcal F}(Q' \cap M)e_{\mathcal F}$ for all finite subsets $\mathcal F \subset I$, with $e_{\mathcal F} = \sum_{i \in \mathcal F} p_i \in (Q' \cap M^\varphi)e$.

Assume by contradiction that $(Q' \cap M^\omega)e \neq (Q' \cap M)e$. Hence there exists a finite subset $\mathcal F \subset I$ such that $e_{\mathcal F}(Q' \cap M^\omega)e_{\mathcal F} \neq e_{\mathcal F}(Q' \cap M)e_{\mathcal F}$. For notational convenience, put $q = e_{\mathcal F}$ and $\mathcal Q = Q' \cap M^\omega$. Observe that $q \mathcal Q q$ is discrete and finite. Let $\varphi_q$ be the faithful normal state on $qMq$ defined by $\varphi_q = \frac{\varphi(q \cdot q)}{\varphi(q)}$. Recall that $q \in Q' \cap M^\varphi$. Denote by $\rE_{qMq} : (qMq)^\omega \to qMq$ the canonical faithful normal conditional expectation. Recall  that $\varphi_q \circ \rE_{qMq} = \varphi_q^\omega$. Since $q\mathcal Qq = (Q q)' \cap (q M q)^\omega$ is discrete, finite and hence of type ${\rm I}$, the faithful normal state $\varphi_q^\omega$ restricted to $q\mathcal Qq$ is diagonalizable. For every eigenvalue $\lambda > 0$, we will denote by 
\begin{equation*}
  \mathcal E_\lambda
  =
  \bigl \{ (x_n)^\omega \in q\mathcal Q q \amid (x_n)^\omega \varphi_q^\omega = \lambda \varphi_q^\omega (x_n)^\omega \bigr \}
\end{equation*}
the spectral subspace of $(q\mathcal Q q, \varphi_q^\omega)$ corresponding to the eigenvalue $\lambda$. Since $q\mathcal Qq \neq q(Q' \cap M)q$, we may choose an eigenvalue $\lambda > 0$ and a nonzero element $f \in \Ball(\mathcal E_\lambda)$ such that $f \notin q(Q' \cap M)q$. Since $f \notin q(Q' \cap M)q$, we have $\|f - \rE_{qMq}(f)\|_{\varphi_q^\omega} > 0$. Moreover, for all $y \in \Ball(qMq)$, we have
\begin{equation*}
  \|f - y\|_{\varphi_q^\omega} \geq \|f - \rE_{qMq}(f)\|_{\varphi_q^\omega} > 0
  \eqstop
\end{equation*}

Put $\varepsilon = \frac12 \|f - \rE_{qMq}(f)\|_{\varphi_q^\omega}$ and $f_1 = f \in \Ball(\mathcal E_\lambda)$. Next, we construct by induction a sequence of elements $f_m \in \Ball(\mathcal E_\lambda)$ such that $\|f_m - f_p\|_{\varphi_q^\omega} \geq \varepsilon$ for all $m, p \geq 1$ such that $m \neq p$. Assume that $f_1, \dots, f_{m} \in \Ball(\mathcal E_\lambda)$ have been constructed. For every $1 \leq j \leq m$, choose a sequence ${(f_{j, n})_n \in \mathcal M^\omega(qMq)}$ such that ${f_{j, n} \in \Ball(q M q)}$ for all $n \in \NN$ and $f_j = (f_{j,n})^\omega$. Let $(x_i)_{i \in \NN}$ be a $\|\cdot\|_{\varphi_q}^\#$-dense sequence in $\Ball(Qq)$. Since $f = f_1 = (f_{1, n})^\omega \in \mathcal E_\lambda$, $\lim_{k \to \omega} \|f_{1, k} x_i - x_i f_{1, k}\|_{\varphi_q}^\# = 0$ for all $i \in \NN$ and $\lim_{k \to \omega} \|f_{1, k} - f_{j, n}\|_{\varphi_q} =  \|f - f_{j, n}\|_{\varphi_q^\omega} \geq 2 \varepsilon$ for all $1 \leq j \leq m$ and all $n \in \NN$, we can find an increasing sequence $(k_n)_n$ in $\NN$ such that for every $n \geq 1$,  we have
\begin{enumerate}
\item [(P1)] $\|f_{1, k_n} \varphi_q - \lambda \varphi_q f_{1, k_n}\| \leq \frac1n$, 
\item [(P2)] $\|f_{1, k_n} x_i - x_i f_{1, k_n}\|_{\varphi_q}^\# \leq \frac1n$ for all $1 \leq i \leq n$ and
\item [(P3)] $\|f_{1, k_n} - f_{j, n}\|_{\varphi_q} \geq \varepsilon$ for all $1 \leq j \leq m$.
\end{enumerate}

By the same reasoning as before, Properties (P1) and (P2) together with Proposition \ref{prop:representation-by-projections} imply that $(f_{1, k_n})_n \in \mathcal M^\omega(qMq)$ and $(f_{1, k_n})^\omega \in \Ball(\mathcal E_\lambda)$. Moreover, Property (P3) implies that $\| (f_{1, k_n})^\omega - f_j\|_{\varphi_e^\omega} \geq \varepsilon$ for all $1 \leq j \leq m$. We can now put $f_{m + 1} = (f_{1, k_n})^\omega$. This finishes the proof of the induction.

So, we have constructed a sequence of elements $f_m \in \Ball(\mathcal E_\lambda)$ such that $\|f_m - f_p\|_{\varphi_q^\omega} \geq \varepsilon$ for all $m, p \geq 1$ such that $m \neq p$. However, since $q\mathcal Qq$ is discrete and finite and hence a countable direct sum of finite dimensional factors, $\Ball(\mathcal E_\lambda)$ is $\|\cdot\|_{\varphi_q^\omega}$-compact and hence we have obtained a contradiction. This finishes the proof of Theorem \ref{thm:diffuse-centraliser}.
\end{proof}

Following \cite{connes74-almost-periodic}, we define the {\em asymptotic centraliser} $M_\omega$ of the von Neumann algebra $M$ by
\begin{equation*}
  M_\omega
  =
    \bigl \{ (x_n)_n \in \linfty(\NN, M) \amid \forall \psi \in M_*, \lim_{n \to \omega} \|x_n \psi - \psi x_n\| = 0 \bigr \} / \mathcal I_\omega(M)
  \eqstop
\end{equation*}
By \cite[Proposition 2.8]{connes74-almost-periodic}, we have $M_\omega = (M' \cap M^\omega)^{\vphi^\omega}$ for every faithful normal state $\vphi \in M_*$. 

\begin{corollary}
  \label{cor:asymptotic-centraliser-diffuse}
  Let $M$ be any factor with separable predual such that $M' \cap M^\omega \neq \CC 1$.  Then $M_\omega$ is diffuse.
\end{corollary}

\begin{proof}
Let $\varphi \in M_\ast$ be any faithful normal state. We have $M_\omega = (M' \cap M^\omega)^{\varphi^\omega} = M' \cap (M^\omega)^{\varphi^\omega}$. Since $M_\omega \neq \CC 1$ and since $M' \cap M = \CC 1$, the projection $z$ obtained in Theorem \ref{thm:diffuse-centraliser} satisfies $z = 0$ and hence $M_\omega$ is diffuse.
\end{proof}

\subsection{The Akemann-Ostrand property (AO)}
\label{sec:AO}

The Akemann-Ostrand property for von Neumann algebras arises from the work of Akemann-Ostrand \cite{akemannostrand75} and was introduced by Ozawa in \cite{ozawa04-solid}. A von Neumann algebra $M \subset \mathcal B(H)$ has {\em property (AO)} if there are unital $\sigma$-weakly dense \Cstar-subalgebras $B \subset M$ and $C \subset M'$ such that $B$ is locally reflexive and such that the map
\begin{equation*}
  B \ot_{\mathrm{alg}} C \ra \bo(H)/\ko(H): 
  x \ot y \mapsto xy
\end{equation*}
is continuous with respect to the minimal tensor \Cstar-norm. We recall the following well-known result. For a proof, we refer the reader to \cite[Chapter 4]{houdayer07}.

\begin{proposition}
Any free Araki-Woods factor satisfies property (AO).
\end{proposition}

\subsection{Intertwining-by-bimodules techniques}
\label{sec:intertwining}

Popa introduced his powerful intertwining-by-bimodule techniques in \cite{popa06-betti-numbers,popa06-strong-rigidity-1,popa06-strong-rigidity-2}. We first recall the intertwining-by-bimodule criterion in the case of {\em finite} von Neumann algebras. Let $(M, \tau)$ be any tracial von Neumann algebra together with von Neumann subalgebras $A \subset 1_A M 1_A$ and $B \subset 1_B M 1_B$. Denote by $\rE_B : 1_B M 1_B \to B$ the unique trace preserving faithful normal conditional expectation. Then the following statements are equivalent (see \cite[Lemma 2.1 and Corollary 2.3]{popa06-strong-rigidity-1}):
  \begin{itemize}
  \item There is $n \geq 1$, a nonzero partial isometry $v \in \mathbb M_{1, n}(1_A M 1_B)$ and a possibly non-unital normal $\ast$-homomorphism $\pi: A \ra \mathbb M_n(B)$ such that $av = v \pi(a)$ for all $a \in A$.
  \item There is no net of unitaries $(w_i)_i$ in $\cU(A)$ such that $\rE_B(x^* w_i y) \ra 0$ $\ast$-strongly as $i \ra \infty$ for all $x, y \in 1_A M 1_B$.
  \end{itemize}
We will say that $A$ {\em embeds into} $B$ {\em inside} $M$ and write $A \preceq_M B$ if one of the above equivalent conditions is satisfied.

Let $(\mathcal M, \Tr)$ be any semifinite von Neumann algebra endowed with a semifinite faithful normal trace. Let $\mathcal B \subset \mathcal M$ be any von Neumann subalgebra such that $\Tr | \mathcal B$ is semifinite. Denote by $\rE_{\mathcal B} : \mathcal M \to \mathcal B$ the unique trace preserving faithful normal conditional expectation. Let $p \in \mathcal M$ be any nonzero finite trace projection and $\mathcal A \subset p\mathcal Mp$ any von Neumann subalgebra. Let $q \in \mathcal B$ be any nonzero finite trace projection. Observe that $p \vee q$ is a nonzero finite trace projection in $\mathcal M$. We will say that $\mathcal A$ {\em embeds into} $q \mathcal B q$ {\em inside} $\mathcal M$ and write $\mathcal A \preceq_{\mathcal M} q \mathcal B q$ if $\mathcal A \preceq_{(p \vee q) \mathcal M (p \vee q)} q \mathcal B q$ in the usual sense for finite von Neumann algebras.

We will need the following useful intertwining-by-bimodule criterion for semifinite von Neumann algebras (see \cite[Lemma 2.2]{houdayerricard11-araki-woods} or \cite[Lemma 2.3]{boutonnethoudayerraum12}).
\begin{lemma}\label{intertwining-general}
Let $(\mathcal M, \Tr)$ be any semifinite von Neumann algebra endowed with a semifinite faithful normal trace. Let $\mathcal B \subset \mathcal M$ be any von Neumann subalgebra such that $\Tr | \mathcal B$ is semifinite. Denote by $\rE_{\mathcal B} : \mathcal M \to \mathcal B$ the unique trace preserving faithful normal conditional expectation.

Let $p \in \mathcal M$ be any nonzero finite trace projection and $\mathcal A \subset p \mathcal M p$ any von Neumann subalgebra. The following conditions are equivalent:
\begin{enumerate}
\item For every nonzero finite trace projection $q \in \mathcal B$, we have $\mathcal A \npreceq_{\mathcal M} q \mathcal B q$.
\item There exists an increasing sequence of nonzero finite trace projections $q_n \in \mathcal B$ such that $q_n \to 1$ strongly and $\mathcal A \npreceq_{\mathcal M} q_n \mathcal B q_n$ for all $n \in \NN$.
\item There exists a net of unitaries $(w_i)$ in $\mathcal U(\mathcal A)$ such that $\lim_k \|\rE_{\mathcal B}(x^* w_i y)\|_{2, \Tr} = 0$ for all $x, y \in p \mathcal M$.
\end{enumerate}
\end{lemma}

\subsection{Deformation/Rigidity theory for free Araki-Woods factors}
\label{sec:deformation-rigidity}

We introduce the s-malleable deformation of free Araki-Woods factors.  It is an analogue of the  malleable deformations considered in \cite{popa06-non-commutative-bernoulli-shifts,ioanapetersonpopa08}. 

Let $U : \RR \ra \cO(H_\RR)$ be any orthogonal representation on a separable real Hilbert space. Denote by $(M, \varphi) = (\Gamma(H_\RR, U_t)\dpr, \varphi_U)$ the associated free Araki-Woods factor together with its free quasi-free state $\varphi$. Put $\cM = \core_\varphi(M)$ and simply denote by $\Tr$ the canonical semifinite faithful normal trace on $\mathcal M$.  Furthermore, we write $(\widetilde{M}, \widetilde \varphi) = (\Gamma(H_\RR \oplus H_\RR, U_t \oplus U_t)\dpr, \varphi_{U \oplus U})$ and $\widetilde{\cM} = \core_{\widetilde \varphi}(\widetilde M)$.  By \cite{shlyakhtenko97}, there are $\ast$-isomorphisms $(\widetilde{M}, \widetilde \varphi) \cong (M, \varphi) \ast (M, \varphi)$ and $\widetilde{\cM} \cong \cM \ast_{\rL_{\varphi}(\RR)} \cM$.  We will identify $M$ with its first copy in $\widetilde{M}$ and $\cM$ with its first copy in $\widetilde{\cM}$.

The orthogonal representation $V : \RR \to \mathcal{O}(H_\RR \oplus H_\RR)$ given by
\begin{equation*}
  V_s = 
  \left (
    \begin{matrix}
      \cos(\frac{\pi}{2} s) & - \sin(\frac{\pi}{2} s) \\
      \sin(\frac{\pi}{2} s) & \cos(\frac{\pi}{2} s)
    \end{matrix}
  \right )
\end{equation*}
commutes with the orthogonal representation $U \oplus U : \RR \to \mathcal O(H_\RR \oplus H_\RR)$.  Hence the associated transformation $\Gamma(V_s)$ on the free Fock space of $H_\RR \oplus H_\RR$ induces a $\ast$-automorphism $\alpha_s$ of $\widetilde{M}$. It satisfies $\alpha_1(x \ast 1) = 1 \ast x$ for all $x \in M$.

Since $\Gamma(V_s)$ fixes the vacuum vector, it preserves the free quasi-free state $\widetilde{\vphi}$ on $\widetilde{M}$.  Hence it induces a trace preserving $\ast$-automorphism of $\widetilde{\cM}$ that we still denote by $\alpha_s$. Likewise, the orthogonal transformation
\begin{equation*}
  \left (
    \begin{matrix}
      1 & 0 \\
      0 & -1
    \end{matrix}
  \right )
\end{equation*}
induces a trace preserving $\ast$-automorphism $\beta$ of $\widetilde{\cM}$ which moreover satisfies $\beta^2 = \id_{\widetilde{\cM}}$, $\beta|_{\cM} = \id_{\cM}$ and $\beta \alpha_s = \alpha_{-s} \beta$ for all $s \in \RR$. Therefore, the deformation $(\alpha_s, \beta)_{s \in \RR}$ is {\em s-malleable} in the sense of Popa \cite{popa06-strong-rigidity-1} and satisfies the following transversality property.
\begin{proposition}[See {\cite[Proposition 4.2]{houdayerricard11-araki-woods}} and {\cite[Lemma 2.1]{popa08-spectral-gap}}]
  \label{prop:transversality}
  \begin{equation*}
    \|x - \alpha_{2s}(x)\|_2 \leq \sqrt 2 \, \|\alpha_s(x) - (\rE_{\mathcal M} \circ \alpha_s)(x)\|_2
    \quad
    \text{ for all } x \in \rL^2(\cM, \Tr)
    \text{ and all } s \in \RR
    \text{.}
  \end{equation*}
\end{proposition}

\section{Proofs of Theorems \ref{thm:omega-solidity} and \ref{thm:core-omega-solid}}
\label{sec:omega-solidity}

\subsection{Preliminaries on $\omega$-solidity}

\begin{definition}
Let $\omega \in \beta(\NN) \setminus \NN$ be a non-principal ultrafilter. We will say that a diffuse von Neumann algebra $M$ is $\omega$-{\em solid} if for every von Neumann subalgebra $Q \subset M$ with expectation whose relative commutant $Q' \cap M^\omega$ is diffuse, we have that $Q$ is amenable. 
\end{definition}

We first show that $\omega$-solidity is stable under amplifications.

\begin{proposition}
  \label{prop:amplifications-and-omega-solidity}
  Let $M$ be any diffuse $\omega$-solid von Neumann algebra.  Then $p(M \vnt \bo(\ltwo))p$ is $\omega$-solid for every nonzero projection $p \in M \vnt \bo(\ltwo)$.
\end{proposition}
\begin{proof}
If $M$ is a diffuse $\omega$-solid von Neumann algebra, then $pMp$ is $\omega$-solid for every nonzero projection $p \in M$. Indeed, let $Q \subset pMp$ be any von Neumann subalgebra such that $Q' \cap (pMp)^\omega$ is diffuse. Put $\mathcal Q = Q \oplus \CC(1 - p)$. Then $\mathcal Q' \cap M^\omega \supset Q' \cap (pMp)^\omega \oplus (1 - p) M^\omega (1 - p)$ is diffuse. Thus $\mathcal Q$ is  amenable and so is $Q$. 

It remains to prove that if $M$ is a diffuse $\omega$-solid von Neumann algebra, then $M \vnt \bo(\ltwo)$ is $\omega$-solid. We may assume that $M$ is not amenable. Observe that if $M$ is properly infinite, then $M \vnt \bo(\ltwo) \cong M$. Since any von Neumann algebra is the direct sum of a finite von Neumann algebra and a properly infinite von Neumann algebra, after cutting down by a central projection, we may assume that $M$ is a diffuse $\omega$-solid finite von Neumann algebra. Since $M$ is the direct sum of a diffuse amenable von Neumann algebra and at most countably many non-amenable $\omega$-solid ${\rm II_1}$ factors, after cutting down by a central projection, we may further assume without loss of generality that $M$ is a non-amenable $\omega$-solid ${\rm II_1}$~factor.

We first prove that $M^t$ is $\omega$-solid for all $t > 0$. Using the first paragraph of the proof, it suffices to prove that $M \vnt \mathbb M_n(\CC)$ is $\omega$-solid for all $n \geq 1$. Let $Q \subset  M \vnt \mathbb M_n(\CC)$ be any von Neumann subalgebra such that $Q' \cap (M \vnt \mathbb M_n(\CC))^\omega$ is diffuse. Assume by contradiction that $Q$ is not amenable. We may choose a projection $q \in Q$ such that $qQq$ is not amenable and $(\tau \otimes \Tr_n)(q) \leq 1/n$.  Note that $(qQq)' \cap (q(M \vnt \mathbb M_n(\CC))q)^\omega = (Q' \cap (M \vnt \mathbb M_n(\CC))^\omega)q$ is diffuse. Regarding $q(M \vnt \mathbb M_n(\CC))q$ as a corner of $M$, we obtain that $M$ is not $\omega$-solid. This is a contradiction.

We now prove that $\mathcal M = M \vnt \bo(\ltwo)$ is $\omega$-solid. Let $Q \subset \mathcal M$ be any von Neumann subalgebra with faithful normal conditional expectation $\rE_Q: \mathcal M \to Q$ and such that $Q' \cap \mathcal M^\omega$ is diffuse. Denote by $\rE_\omega : \mathcal M^\omega \to \mathcal M$ the canonical faithful normal conditional expectation. Let $\vphi \in \mathcal  M_\ast$ be any faithful normal state such that $\varphi \circ \rE_Q = \varphi$.  By Theorem~\ref{thm:diffuse-centraliser}, the relative commutant $Q' \cap (\mathcal M^\omega)^{\vphi^\omega}$ is diffuse as well. 

Fix a tracial faithful normal semifinite weight $\Tr$ on $\mathcal M$. By \cite[Lemma 4.26]{andohaagerup12}, $\Tr^\omega = \Tr \circ \rE_\omega$ is a tracial faithful normal semifinite weight on $\mathcal M^\omega$. Denote  by $T \in \Lone(\mathcal M, \Tr)_+$ the unique positive selfadjoint operator affiliated with $\mathcal M$ satisfying $\vphi = \Tr(T \, \cdot)$.  By \cite[Lemme 1.2.3 (b) and Lemme 1.4.4]{connes73-type-III}, we have $\varphi^\omega = \Tr^\omega(T \, \cdot)$.

Denote by $B \subset \mathcal M$ the von Neumann subalgebra generated by all the spectral projections of $T$. Put $\mathcal Q = Q \vee B$. Since $\varphi^\omega = \Tr^\omega(T \, \cdot)$, we have $\mathcal M^{\varphi^\omega} = B' \cap \mathcal M^\omega$ and hence
  \begin{equation*}
    \mathcal Q' \cap \mathcal M^\omega
    =
    Q' \cap B' \cap \mathcal M^\omega
    =
    Q' \cap (\mathcal M^\omega)^{\vphi^\omega}
  \end{equation*}
  is diffuse. Observe that $\mathcal Q$ is globally invariant under the modular automorphism group $(\sigma_t^\varphi)$. Since $T \in \Lone(\mathcal M, \Tr)$, we may choose a sequence of finite trace projections $p_k \in B \subset \mathcal Q$ such that  $p_k \to 1$ strongly. Since $p_k \mathcal M p_k$ is an $\omega$-solid ${\rm II_1}$ factor by the first part of the proof and since $(p_k\mathcal Q p_k)' \cap (p_k \mathcal M p_k)^\omega = (\mathcal Q' \cap \mathcal M^\omega)p_k = (Q' \cap (\mathcal M^\omega)^{\vphi^\omega})p_k$ is diffuse, we have that $p_k \mathcal Q p_k$ is amenable. Since amenability is stable under direct limits, we finally obtain that $\mathcal Q$ is amenable. Since $Q \subset \mathcal Q$ is a von Neumann subalgebra with expectation, $Q$ is also amenable. This finishes the proof of Proposition \ref{prop:amplifications-and-omega-solidity}.
\end{proof}

Next, we prove a useful characterisation of $\omega$-solidity. 

\begin{proposition}
\label{prop:omega-solid-characterisation}
Let $M$ be any von Neumann algebra with separable predual that has no amenable direct summand. The following conditions are equivalent.
\begin{enumerate}
\item For every von Neumann subalgebra $Q \subset M$ with expectation, if $Q' \cap M^\omega$ is diffuse then $Q$ is amenable.
\item For every von Neumann subalgebra $Q \subset M$ with expectation that has no amenable direct summand, the relative commutant $Q' \cap M^\omega$ is discrete.
\end{enumerate}
\end{proposition}

\begin{proof}
(i) $\Rightarrow$ (ii). Let $Q \subset M$ be any von Neumann subalgebra with expectation that has no amenable direct summand. By Theorem \ref{thm:diffuse-centraliser}, there is a unique central projection $z \in \mathcal Z(Q' \cap M^\omega) \cap \mathcal Z(Q' \cap M)$ such that $(Q' \cap M^\omega)z$ is diffuse and $(Q' \cap M^\omega)(1 - z)$ is discrete. Put $\mathcal Q = Q z \oplus \CC(1 - z) $. Since $\mathcal Q' \cap M^\omega \supset (Q' \cap M^\omega)z \oplus (1 - z) M^\omega (1 - z)$ is diffuse, we have that $\mathcal Q$ is amenable. Thus $z = 0$ and $Q' \cap M^\omega$ is discrete.

(ii) $\Rightarrow$ (i). Let $Q \subset M$ be any von Neumann subalgebra with expectation such that $Q' \cap M^\omega$ is diffuse. Denote by $z \in \mathcal Z(Q)$ the unique central projection such that $Qz$ has no amenable direct summand and $Q(1 - z)$ is amenable. Since $(1 - z) M (1 - z)$ has no amenable direct summand, $\mathcal Q = Q z \oplus (1 - z)M(1 - z)$ has no amenable direct summand either.  Then $\mathcal Q' \cap M^\omega$ is discrete and so is $(Q' \cap M^\omega)z = (\mathcal Q' \cap M^\omega)z$. Thus $z = 0$ and $Q$ is amenable. 
\end{proof}

\subsection{Proof of Theorem \ref{thm:omega-solidity}}
\label{sec:omega-solidity-property-AO}

\begin{proof}[Proof of Theorem~\ref{thm:omega-solidity}]
Let $M$ be any von Neumann algebra with separable predual that satisfies property (AO). Denote by $\rE_\omega : M^\omega \to M$ the canonical faithful normal conditional expectation.   Let $Q \subset M$ be any von Neumann subalgebra with faithful normal conditional expectation $\rE_Q: M \ra Q$ and such that $Q' \cap M^\omega$ is diffuse. Fix a faithful normal state $\varphi \in M_\ast$ satisfying $\varphi \circ \rE_Q = \varphi$. By Theorem \ref{thm:diffuse-centraliser}, $Q' \cap (M^\omega)^{\vphi^\omega}$ is diffuse and hence there is a sequence of unitaries $U_k$ in $\mathcal U(Q' \cap (M^\omega)^{\vphi^\omega})$ such that $U_k \ra 0$ weakly as $k \ra \infty$. Choose a sequence $(u_m^k)_m \in \mathcal M^\omega(M)$ such that $u_m^k \in \Ball(M)$ for all $m \in \NN$ and $U_k = (u_m^k)^\omega$. Let $(x_i)_{i \geq 1}$ be a $\|\cdot\|_\varphi^\#$-dense sequence in $\Ball(Q)$ and $(\psi_j)_{j \geq 1}$ be a $\| \cdot \|$-dense sequence in $M_*$. 
  
 There exists an increasing sequence $(k_n)_n$ in $\NN$ such that for every $n \in \NN$, we have $\lim_{m \to \omega} |\psi_j(u^{k_n}_m)| = |(\psi_j \circ \rE_\omega)(U_{k_n})| < \frac1n$ for all $1 \leq j \leq n$. Therefore, there exists an increasing sequence $(m_n)_n$ in $\NN$ such that for every $n \in \NN$, the element $u_n = u_{m_n}^{k_n} \in \Ball(M)$ satisfies 
 \begin{enumerate}
\item [(P1)] $\| u_n \vphi - \varphi u_n\| \leq \frac{1}{n}$, 
\item [(P2)] $\|1 - u_n^*u_n\|_\varphi^\# \leq \frac1n$ and $\|1 - u_n u_n^*\|_\varphi^\# \leq \frac1n$
\item [(P3)] $\|u_n x_i - x_i u_n\|_{\varphi}^\# \leq \frac1n$ for all $1 \leq i \leq n$ and
\item [(P4)] $|\psi_j(u_n)| \leq \frac{1}{n}$ for all $1 \leq j \leq n$.
\end{enumerate}
Property (P1) and Proposition \ref{prop:representation-by-projections} show that $(u_n)_n \in \mathcal M^\omega(M)$ and together with Properties (P2) and (P3) they show that $U = (u_n)^\omega \in \mathcal U(Q' \cap (M^\omega)^{\vphi^\omega})$. Finally, Property (P4) shows that $u_n \ra 0$ weakly as $n \ra \infty$.

 Regard $M \subset \mathcal B(H)$ where the Hilbert space $H$ is given by property (AO). Define the unital completely positive map $\Theta : \bo(H) \to \bo(H)$ by $\Theta(T) = \sigma\text{-weak} \lim_{n \to \omega} u_n T u_n^*$. Observe that $\Theta(x) = \rE_\omega(U x U^*) \in M$ for all $x \in M$ and hence $\Theta|_M$ is normal. Next, define the unital completely positive maps
  \begin{equation*}
    \Psi_k
    =
    \frac{1}{k} \sum_{j = 1}^k \Theta^{\circ j} : \bo(H) \to \bo(H)
  \end{equation*}
  and let $\Psi$ be the pointwise $\sigma$-weak limit of $(\Psi_k)_k$, as $k \ra \omega$.  Since $\Theta(M) \subset M$, we have $\Psi_k(M) \subset M$ for all $k \geq 1$ and hence $\Psi(M) \subset M$.  Note that since $\Theta|_M$ is normal, we have $\Psi|_M = \Theta \circ \Psi|_M$.  Since $U \in \mathcal U((M^\omega)^{\vphi^\omega})$, we also have $\vphi \circ \Theta|_M = \vphi$. This implies that $\vphi \circ \Psi_k|_M = \vphi$ for all $k \geq 1$ and hence $\vphi \circ \Psi|_M = \vphi$.

  Put $\cQ = \{U, U^*\}' \cap M$. Observe that $Q \subset \cQ \subset M$ and that $\mathcal Q$ is globally invariant under the modular automorphism group $(\sigma_t^\vphi)$.  Let $x \in M$ and put $y = \Psi(x) \in M$.  We have $y = \Psi(x) = \Theta( \Psi(x) ) = \Theta(y) = \rE_\omega(U y U^*)$.  Since $U \in \mathcal U((M^\omega)^{\vphi^\omega})$ and since $\| y \|_{\vphi^\omega} = \|y\|_\vphi$, we obtain $\|Uy U^*\|_{\vphi^\omega} = \|y \|_\vphi$ and hence
  \begin{align*}
    \|y - U y U^*\|_{\vphi^\omega}^2
    &= \|y\|_{\vphi^\omega}^2 + \|U y U^*\|_{\vphi^\omega}^2 - 2 \, \mathrm{Re} \, \varphi^\omega( U y^* U^*y)
    \\
    &= 2 \|y\|_\vphi^2 - 2 \, \mathrm{Re} \, \varphi(\rE_\omega( U y^* U^* y))
    \\
    &= 2 \|y\|_\vphi^2 - 2 \, \mathrm{Re} \, \varphi(\rE_\omega( U y^* U^*)y)
    \\
    &= 2 \|y\|_\vphi^2 - 2 \, \mathrm{Re} \, \varphi(y^*y) = 0
    \eqstop
  \end{align*}
  Therefore, we have $y = U y U^*$ and hence $\Psi(x) = y \in \cQ$. Combining this with the fact that $\Theta(x) = x$ for all $x \in \cQ$, we see that $\Psi|_M$ is a norm one projection onto $\cQ$.  We already saw that $\Psi|_M$ is $\vphi$-preserving and hence we infer that $\Psi|_M = \rE_\cQ : M \to \mathcal Q$ is the unique $\vphi$-preserving conditional expectation from $M$ onto~$\cQ$.

  Define $\Phi_\cQ : M \ot_{\mathrm{alg}} M' \to \bo(H) : \sum_{j = 1}^n b_i \ot  c_i  \mapsto \sum_{j = 1}^n \rE_\cQ(b_i) \, c_i$. By definition of $\Psi$, we have $\Psi(c) = c$ for all $c \in M'$. Therefore \cite[Theorem 3.1]{choi74} implies that for all $n \geq 1$, all $b_i \in M$ and all $c_i \in M'$, we have 
  \begin{equation*}
    \Psi(\sum_{j = 1}^n b_i \, c_i)
    =
    \sum_{j = 1}^n \Psi(b_i) \, c_i
    =
    \sum_{j = 1}^n \rE_\cQ(b_i) \, c_i
    =
    \Phi_\cQ(\sum_{j = 1}^n b_i \otimes c_i)
    \eqstop
  \end{equation*}

  The fact that $u_n \to 0$ $\sigma$-weakly as $n \to \omega$ implies that $\Theta(T) = \sigma\text{-weak} \lim_{n \to \omega} u_n T u_n^*= 0$ for all $T \in \ko(H)$.  This shows that $\Psi(T) = 0$ for all $T \in \ko(H)$.  Hence $\ko(H) \subset \ker \Psi$.  Denote by ${\pi : \bo(H) \to \bo(H) / \ko(H)}$ the canonical quotient map.  Then there exists a unital completely positive map ${\widetilde \Psi : \bo(H) / \ko(H) \to \bo(H)}$ such that $\Psi = \widetilde \Psi \circ \pi$.

  By property (AO) of $M$, there is a unital $\sigma$-weakly dense locally reflexive \Cstar-subalgebra $B \subset M$ and a unital $\sigma$-weakly dense \Cstar-subalgebra $C \subset M'$ together with a $\ast$-homomorphism 
  \begin{equation*}
    \nu :
    B \ot_{\mathrm{alg}} C \to \bo(H) / \ko(H) :
    \sum_{j = 1}^n b_i \otimes c_i \mapsto \pi(\sum_{j = 1}^n b_i \, c_i)
  \end{equation*}
  that is continuous with respect to the minimal tensor norm on $B \ot_{\mathrm{alg}} C$. Therefore $\Phi_\cQ = \widetilde \Psi \circ \nu$ is continuous with respect to the minimal tensor norm on $B \ot_{\mathrm{alg}} C$.  Applying \cite[Lemma~5]{ozawa04-solid}, we obtain that $\cQ$ is amenable and so is $Q$.  
\end{proof}

\subsection{Proof of Theorem \ref{thm:core-omega-solid}}
\label{sec:omega-solidity-cores}  

The next theorem is a generalisation of \cite[Theorem 3.4]{houdayer10} regarding the position of the relative commutant of large subalgebras of the continuous core of free Araki-Woods factors in the ultraproduct framework. 

\begin{theorem}
  \label{thm:spectral-gap}
  Let $U : \RR \ra \cO(H_\RR)$ be any orthogonal representation on a separable real Hilbert space. Denote by $(M, \varphi) = (\Gamma(H_\RR, U_t)\dpr, \varphi_U)$ the corresponding free Araki-Woods factor together with its free quasi-free state and by $\mathcal M = \core_\varphi(M)$ the continuous core associated with the free quasi-free state $\varphi$. 

Then for every nonzero finite trace projection $p \in \rL_\varphi(\RR) \subset \mathcal M$ and every von Neumann subalgebra $\mathcal Q \subset p \mathcal M p$ that has no amenable direct summand, there exists a nonzero finite trace projection $q \in \rL_\varphi(\RR)$ such that 
\begin{equation*}
  \mathcal Q' \cap p \mathcal M^\omega p
  \preceq_{\mathcal M^\omega}
  \rL_\varphi(\RR)^\omega q
  \eqstop
\end{equation*}
\end{theorem}

\begin{proof}[Proof of Theorem~\ref{thm:spectral-gap}]
The proof is very much inspired by \cite[Theorems 4.3 and 4.5]{peterson06} (see also \cite[Theorem D]{houdayer12-structure}).  Let $\alpha_t: \widetilde{\cM} \rightarrow \widetilde{\cM}$ be the trace preserving s-malleable deformation introduced in Subsection~\ref{sec:deformation-rigidity}.  Write $\widetilde{\cM} = \cM *_{\rL_\varphi(\RR)} \alpha_1(\cM)$. Observe that if $(x_n)_n \in \mathcal I_\omega(\widetilde{\mathcal M})$, then also $(\alpha_t(x_n))_n \in \mathcal I_\omega(\widetilde{ \mathcal M})$ for all $t \in \RR$.  It follows that $(\alpha_t)$ extends to a one-parameter family of trace preserving $\ast$-automorphisms of the ultraproduct von Neumann algebra $\widetilde{\cM}^\omega$ that we denote by $(\alpha_t^\omega)$.  We emphasise however that $t \mapsto \alpha_t(x)$ need not be continuous when $x \in \widetilde{\cM}^\omega$.

  \textbf{Step 1: Uniform convergence in $\|\cdot\|_2$ of $(\alpha_t^\omega)$ on $\Ball(\mathcal{Q}' \cap p \cM^\omega p)$.}  Assume by contradiction that $(\alpha_t^\omega)$ does not converge uniformly in $\|\cdot\|_2$ on $\Ball(\mathcal{Q}' \cap p\cM^\omega p)$. Thus there exist $c > 0$, a sequence $(t_k)_k$ of positive reals such that $\lim_k t_k = 0$ and a sequence $(X_k)_k$ in $\Ball(\mathcal{Q}' \cap p \cM^\omega p)$ such that $\|X_k - \alpha_{2t_k}^\omega(X_k)\|_2 \geq 2c$ for all $k \in \NN$. Write $X_k = (x_{k, n})^\omega$ with $x_{k, n} \in \Ball(p \mathcal M p)$ satisfying $\lim_{n \to \omega} \|y x_{k, n} - x_{k, n} y\|_2 = 0$ and $\|X_k - \alpha_{2t_k}^\omega(X_k)\|_2 = \lim_{n \to \omega} \|x_{k, n} - \alpha_{2t_k}(x_{k, n})\|_2$ for all $k \in \NN$ and all $y \in \mathcal Q$.
  
  Denote by $I$ the directed set of all pairs $(\mathcal F, \varepsilon)$ with $\varepsilon  > 0$ and $\mathcal F \subset \Ball(\mathcal Q)$ finite subset. Let $i = (\mathcal F, \varepsilon) \in I$. Choose $k \in \NN$ large enough so that $\|a - \alpha_{t_k}(a)\|_2 \leq \varepsilon/3$ for all $a \in \mathcal F$. Then choose $n \in \NN$ large enough so that $\|x_{k, n} - \alpha_{2t_k}(x_{k, n})\|_2 \geq c$ and $\|a x_{k, n} - x_{k, n} a\|_2 \leq \varepsilon /3$ for all $a \in \mathcal F$.
  
  Put $\xi_i = \alpha_{t_k}(x_{k, n}) - \rE_{p \mathcal M p}(\alpha_{t_k}(x_{k, n})) \in \rL^2(p \widetilde{\mathcal M} p) \ominus \rL^2(p \mathcal M p)$. By Proposition \ref{prop:transversality}, we have
  \begin{equation*}
    \|\xi_i\|_2 \geq \frac{1}{\sqrt{2}} \|x_{k, n} - \alpha_{2t_k} (x_{k, n})\|_2 \geq \frac{c}{\sqrt{2}}
    \eqstop
  \end{equation*}
  For all $x \in p \mathcal M p$, we have
  \begin{equation*}
    \|x \xi_i\|_2 = \|(1 - \rE_{p \mathcal M p}) (x \alpha_{t_k} (x_{k, n})) \|_2 \leq \|x \alpha_{t_k}(x_{k, n})\|_2 \leq \|x\|_2
    \eqstop
  \end{equation*}
  By Popa's spectral gap argument \cite{popa08-spectral-gap}, for all $a \in \mathcal F$, we have
  \begin{align*}
    \| a \xi_i - \xi_i a \|_2
    & =
    \|(1 - \rE_{p \mathcal Mp}) (a \alpha_{t_k}( x_{k, n} ) - \alpha_{t_k}( x_{k, n}) a)\|_2 \leq \|a \alpha_{t_k}( x_{k, n}) - \alpha_{t_k}( x_{k, n}) a\|_2 \\
    & \leq
    2 \|a - \alpha_{t_k}(a)\|_2 + \|a x_{k, n} - x_{k, n} a\|_2 \leq \varepsilon.
  \end{align*}

Hence $\xi_i \in \rL^2(p\widetilde {\mathcal M} p) \ominus \rL^2( p \mathcal Mp)$ is a net of vectors satisfying $\limsup_i \|x \xi_i\|_2 \leq \|x\|_2$ for all $x \in p \mathcal Mp $, $\liminf_i \|\xi_i\|_2 \geq~\frac{c}{\sqrt{2}}$ and $\lim_i \|a \xi_i - \xi_i a\|_2 = 0$ for all $a \in \mathcal Q$. Since the $p \mathcal M p$-$p \mathcal M p$-bimodule $\rL^2(p\widetilde {\mathcal M} p) \ominus \rL^2( p \mathcal Mp)$ is weakly contained in the coarse $p \mathcal M p$-$p \mathcal M p$-bimodule $\rL^2(p \mathcal M p) \otimes \rL^2( p \mathcal Mp)$ (see \cite[Lemma 5.1]{houdayerricard11-araki-woods}), it follows that $\mathcal Q$ has an amenable direct summand by Connes' characterisation of amenability \cite{connes76}. This is a contradiction and hence $(\alpha_t^\omega)$ does converge uniformly in $\|\cdot\|_2$ on $\Ball(\mathcal{Q}' \cap p \cM^\omega p)$.

 We now proceed by contradiction and assume that $\mathcal Q' \cap p \mathcal M^\omega p \npreceq_{ \mathcal M ^\omega} \rL_\varphi(\RR)^\omega q$ for every nonzero finite trace projection $q \in \rL_\varphi(\RR)$. By Lemma \ref{intertwining-general}, there exists a net $(U_k)_k$ of unitaries in $\mathcal U(\mathcal Q' \cap p \mathcal M^\omega p )$ such that $\lim_k \|\rE_{\rL_\varphi(\RR)^\omega}(X^* U_k Y)\|_2 = 0$ for all $X, Y \in p \mathcal M^\omega$.
  
  \textbf{Step 2: Uniform convergence in $\|\cdot\|_2$ of $(\alpha_t)$ on $\Ball(\mathcal{Q})$.} Take $\veps > 0$.  Since $(\alpha_t^\omega)$ converges uniformly in $\|\cdot\|_2$ on $\Ball(\mathcal{Q}' \cap p \cM^\omega p)$, there is some $t_0 > 0$ such that for all $t \in [0, t_0]$, we have $\|\alpha_t^\omega(X) - X\|_2 < \veps^2/4$ for all $X \in \Ball (\mathcal{Q}' \cap p \cM^\omega p)$.  We show that for all $t \in [0, t_0]$ and all $x \in \Ball(\mathcal{Q})$, we have $\|\alpha_t(x) - x\|_2 < \veps$.  
  
Take $t \in [0, t_0]$ and  $x \in \Ball(\mathcal{Q})$.  Let $(y_i)_i$ be a $\|\cdot\|_2$-dense sequence in $\Ball(p \cM )$.  There is an increasing sequence $(k_n)_n$ such that for every $n \geq 1$, the unitary $U_{k_n} \in \mathcal{U}(\mathcal{Q}' \cap p \cM^\omega p)$ satisfies $\|\rE_{\rL_\varphi(\RR)^\omega}(y_i^* U_{k_n} y_j)\|_2 < 1/n$ for all $i,j \in \{1, \dotsc, n\}$. Write $U_{k_n} = (u_m^{k_n})^\omega$ with  $u_m^{k_n} \in \Ball(p \mathcal M p)$ for all $m \in \NN$.  There exists an increasing sequence $(m_n)_n$ in $\NN$ such that for every $n \geq 1$, the element $v_n = u_{m_n}^{k_n} \in \Ball(p \mathcal M p)$ satisfies
  \begin{itemize}
  \item $\|v_n x v_n^* - x \|_2 \leq 1/n$,
  \item $\|\rE_{\rL_\varphi (\RR)}(y_i^* v_n y_j)\|_2 \leq 1/n$ for all $i,j \in \{1, \dotsc, n\}$, and
  \item $\|\alpha_t(v_n) - v_n\|_2 \leq \veps^2/4$.
  \end{itemize}
  Since $(y_i)_i$ is $\|\cdot\|_2$-dense in $\Ball(p \cM)$, the second condition implies that $\|\rE_{\rL_\varphi(\RR)}(a^* v_n b)\|_2 \rightarrow 0$ for all $a,b \in p \cM $.  Writing now $\delta_t(x) = \alpha_t(x) - \rE_{p \cM p} (\alpha_t (x)) \in p \widetilde {\mathcal M} p \ominus p \mathcal M p$, we obtain
  \begin{align*}
    \|\delta_t(x)\|_2^2
    & =
    \langle \delta_t(x), \delta_t(x) \rangle \\
    & \leq
    | \langle \delta_t(v_n x v_n^*), \delta_t(x) \rangle| + \|v_n x v_n^* - x\|_2 \\
    & \leq
    | \langle v_n \delta_t(x) v_n^*, \delta_t(x) \rangle| + \|v_n x v_n^* - x\|_2  + 2 \|v_n - \alpha_t(v_n)\|_2 \\
    & \leq
    | \langle v_n \delta_t(x) v_n^*, \delta_t(x) \rangle| + 1/n + \veps^2/2 
    \eqstop
  \end{align*}
 Observe moreover that by Cauchy-Schwarz inequality, we have
 \begin{align*}
 | \langle v_n \delta_t(x) v_n^*, \delta_t(x) \rangle | &= |\Tr(\delta_t(x)^* v_n \delta_t(x) v_n^*)| \\
 &= |\Tr(\rE_{p \mathcal M p}(\delta_t(x)^* v_n \delta_t(x)) v_n^*)| \\
 & \leq \|\rE_{p \mathcal M p}(\delta_t(x)^* v_n \delta_t(x))\|_2.
\end{align*} 
Since $\delta_t(x) \in p(\widetilde{\mathcal M} \ominus \mathcal M)$ and since $\lim_n \|\rE_{\rL_\varphi(\RR)}(a^* v_n b)\|_2 = 0$ for all $a,b \in p \cM $, by \cite[Theorem 2.5, Claim]{boutonnethoudayerraum12}, it follows that $\lim_n \|\rE_{p \mathcal M p}(\delta_t(x)^* v_n \delta_t(x))\|_2 = 0$ and hence $\lim_n | \langle v_n \delta_t(x) v_n^*, \delta_t(x) \rangle| = 0$. Hence, the transversality property of Proposition~\ref{prop:transversality} now yields $\|x - \alpha_{2t}(x)\|_2 \leq \sqrt{2} \, \|\delta_t(x)\|_2 \leq \veps$.  Thus, $(\alpha_t)$ converges uniformly in $\|\cdot\|_2$ on $\Ball(\mathcal{Q})$.

  \textbf{Step 3: Deducing a contradiction.} Since $(\alpha_t)$ converges uniformly in $\|\cdot\|_2$ on $\Ball(\mathcal{Q})$, \cite[Theorem 4.3]{houdayerricard11-araki-woods} implies that there exists a nonzero finite trace projection $r \in \rL_\varphi(\RR)$ such that $\mathcal Q \preceq_{ \cM } \rL_\varphi(\RR) r$.  Since $\rL_\varphi(\RR)r$ is amenable, it follows that $\mathcal Q$ has an amenable direct summand, contradicting our assumption that it does not.  It follows that the assumption $\mathcal Q' \cap p \mathcal M^\omega p \npreceq_{\mathcal M^\omega} \rL_\varphi(\RR)^\omega q$ for every nonzero finite trace projection $q \in \rL_\varphi(\RR)$ of the previous step is wrong.  This finishes the proof of the theorem.
\end{proof}

Before we can proceed to the proof of Theorem~\ref{thm:core-omega-solid}, we need a few basic results regarding mixing inclusions in semifinite amalgamated free products. Recall that an inclusion of tracial von Neumann algebras $B \subset (M, \tau)$ is {\em mixing} if for every uniformly bounded net $(w_k)_k$ of elements in $B$ that goes to $0$ weakly, we have
\begin{equation*}
  \lim_k \|\rE_B(x w_k y)\|_2 = 0, \; \forall x, y \in M \ominus B
  \eqstop
\end{equation*}

Let now $(\mathcal M, \Tr)$ be any semifinite von Neumann algebra endowed with a semifinite faithful normal trace. Let $\mathcal B \subset \mathcal M$ be any von Neumann subalgebra such that $\Tr |_\mathcal B$ is semifinite. Denote by $\rE_{\mathcal B} : \mathcal M \to \mathcal B$ the unique trace preserving faithful normal conditional expectation.

\begin{definition}\label{def:mixing}
Keep the same notation. We will say that the inclusion $\mathcal B \subset \mathcal M$ is {\em mixing} if for every nonzero finite trace projection $q \in \mathcal B$ and for every uniformly bounded net $(w_k)_k$ in $q\mathcal Bq$ that goes to $0$ weakly, we have
\begin{equation*}
  \lim_k \|\rE_{\mathcal B}(x^* w_k y)\|_2 = 0, \; \forall x, y \in q(\mathcal M \ominus \mathcal B)
  \eqstop
\end{equation*}
\end{definition}

We prove a useful characterisation of mixing inclusions of semifinite von Neumann algebras.

\begin{lemma}
\label{lem:mixing-characterisation}
Keep the same notation. The following conditions are equivalent.
\begin{enumerate}
\item The inclusion $\mathcal B \subset \mathcal M$ is mixing.
\item For every nonzero finite trace projection $q \in \mathcal B$, the inclusion of tracial von Neumann algebras $q \mathcal B q \subset q \mathcal M q$ is mixing.
\item There exists an increasing sequence of nonzero finite trace projections $q_n \in \mathcal B$ such that the inclusion of tracial von Neumann algebras $q_n \mathcal B q_n \subset q_n \mathcal M q_n$ is mixing for all $n \in \NN$.
\end{enumerate}
\end{lemma}
 
\begin{proof}
(i) $\Rightarrow$ (ii) $\Rightarrow$ (iii) are obvious. For (iii) $\Rightarrow$ (i), let $q \in \mathcal B$ be a nonzero finite trace projection, $(w_k)_k$ a net of elements in $\Ball(q \mathcal B q)$ that goes to $0$ weakly and $x, y \in \Ball(\mathcal M) \cap q(\mathcal M \ominus \mathcal B)$. 

Take $\varepsilon > 0$. Since $\Tr(q) < \infty$ and $q_n \to 1$ strongly, there exists $n \in \NN$ such that 
\begin{equation}\label{eq:mixing1}
\|q - q_n q\|_2 + \|q - qq_n\|_2 + \|x^* q - q_n x^* q\|_2 + \|q y - qy q_n \|_2 \leq \frac{\varepsilon}{4}.
\end{equation}
This implies in particular that for all $k$, we have
\begin{equation}\label{eq:mixing2}
\|w_k - q_n w_k q_n\|_2 \leq \|w_k - q_n w_k \|_2 + \|q_n(w_k  - w_k q_n) \|_2 \leq \frac{\varepsilon}{4}.
\end{equation}
Since the inclusion of tracial von Neumann algebras $q_n \mathcal B q_n \subset q_n \mathcal M q_n$ is mixing, since $q_n q y q_n, q_n x^* q q_n \in q_n \mathcal M q_n \ominus q_n \mathcal B q_n$ and since $q_n w_k q_n \to 0$ weakly as $k \to \infty$, there exists $k_0$ such that for all $k \geq k_0$, we have
\begin{equation}\label{eq:mixing3}
  \|\rE_{\mathcal B}(q_n x^* qq_n \,  q_nw_kq_n \, q_nq y q_n)\|_2
  =
  \Tr(q_n)^{1/2} \, \|\rE_{q_n \mathcal B q_n}(q_n x^* qq_n \,  q_nw_kq_n \, q_nq y q_n)\|_{2, \tau_{q_n \mathcal M q_n}} \leq \frac{\varepsilon}{2}.
\end{equation}
Combining (\ref{eq:mixing1}), (\ref{eq:mixing2}) and (\ref{eq:mixing3}), we obtain 
\begin{align*}
\|\rE_{\mathcal B}(x^* w_k y)\|_2 & \leq \|\rE_{\mathcal B}(q_nx^*q \, w_k \,  qyq_n)\|_2 + \frac{\varepsilon}{4} \\
& \leq \|\rE_{\mathcal B}(q_nx^*q \, q_nw_kq_n \,  qyq_n)\|_2 + \frac{\varepsilon}{2} \\
& \leq \varepsilon. \qedhere
\end{align*} 
\end{proof} 
 
An interesting class of mixing inclusions of semifinite von Neumann algebras arises from modular  automorphism groups.

\begin{proposition}\label{prop:mixing-modular}
Let $(M, \varphi)$ be any von Neumann algebra together with a faithful normal state such that the modular automorphism group $(\sigma_t^\varphi)$ is mixing, that is, for all $x, y \in M$, we have $\lim_{|t| \to \infty} \varphi( \sigma_t^{\varphi}(x) y) = \varphi(x)\varphi(y)$.  Denote by $\core_\varphi(M)$ the continuous core associated with $\varphi$. Then the inclusion $\rL_\varphi (\RR) \subset \core_{\varphi}(M)$ is mixing.
\end{proposition}
 
\begin{proof}
By Fourier transform, identify $\rL_\varphi(\RR)$ with $\rL^\infty(\RR)$. The proof of \cite[Theorem 3.7]{houdayer10} shows that the inclusion of tracial von Neumann algebras $\rL_\varphi (\RR) q \subset q \core_\varphi(M) q$ is mixing for all nonzero projections $q$ corresponding to the bounded intervals of the form $[-T, T]$ with $T > 0$. Then Lemma~\ref{lem:mixing-characterisation} shows that the inclusion $\rL_\varphi(\RR) \subset \core_\varphi(M)$ is mixing.
\end{proof} 
 
For all $i \in \{1, 2\}$, let $\mathcal B \subset \mathcal M_i$ be an inclusion of von Neumann algebras with faithful normal conditional expectation $\rE_i : \mathcal M_i \to \mathcal B$. Assume that $\mathcal B$ is semifinite with faithful normal semifinite tracial weight $\Tr$. Assume moreover that $\Tr \circ \rE_i$ is still a semifinite trace on $\mathcal M_i$. Consider the amalgamated free product von Neumann algebra $(\mathcal M, \rE) = (\mathcal M_1, \rE_1) \ast_{\mathcal B} (\mathcal M_2, \rE_2)$ and observe that $\Tr \circ \rE$ is still a faithful normal semifinite trace on $\mathcal M$ (see \cite[Section 2.2]{boutonnethoudayerraum12}). We say in that case that $\mathcal M = \mathcal M_1 \ast_{\mathcal B} \mathcal M_2$ is a {\em semifinite} amalgamated free product von Neumann algebra.

We prove the analogue of \cite[Proposition 4.7]{houdayer12} in the setting of semifinite amalgamated free product von Neumann algebras.
 
\begin{proposition}
  \label{prop:mixing-afp}
Let $\mathcal M = \mathcal M_1 \ast_{\mathcal B} \mathcal M_2$ be a semifinite amalgamated free product von Neumann algebra. Assume that the inclusion $\mathcal B \subset \mathcal M_2$ is mixing.  Then the inclusion $\mathcal M_1  \subset \mathcal M $ is mixing.
\end{proposition}

\begin{proof}
Denote by $\rE_{\mathcal M_1} : \mathcal M \to \mathcal M_1$ the unique trace preserving faithful normal conditional expectation. To prove that the inclusion $\mathcal M_1  \subset \mathcal M $ is mixing, using Kaplansky's density theorem and Lemma \ref{lem:mixing-characterisation}, it suffices to show that for all nonzero finite trace projections $q \in \mathcal B$, all nets $(w_k)_k$ of elements in $\Ball(q \mathcal M_1 q)$ that go to $0$ weakly and all elements $x, y \in q(\mathcal M \ominus \mathcal M_1)$ of the form $x = qx_1 \cdots x_{2m + 1}$ and $y = qy_1 \cdots y_{2n + 1}$ with $m, n \geq 1$, $x_1, x_{2m + 1}, y_1, y_{2n +1} \in \Ball(\mathcal M_1)$, $x_2, \dots, x_{2m}, y_2, \dots, y_{2n} \in \Ball(\mathcal M_2) \cap (\mathcal M_2 \ominus \mathcal B)$ and $x_{3}, \dots, x_{2m - 1}, y_3, \dots, y_{2n - 1} \in \Ball(\mathcal M_1) \cap (\mathcal M_1 \ominus \mathcal B)$, we have 
\begin{equation*}
  \lim_k \|\rE_{\mathcal M_1}(x^* w_k y)\|_2 = 0
  \eqstop
\end{equation*}

 Using the property of freeness with amalgamation over $\mathcal B$, we have
  \begin{align*}
    \rE_{\mathcal M_1}(x^* w_k y) &= 
    \rE_{\mathcal M_1}( x_{2m + 1}^* \cdots x_2^* \, x_1^*  qw_kq  y_1 \, y_2 \cdots y_{2n + 1}) \\
    & =
 \rE_{\mathcal M_1}( x_{2m + 1}^* \cdots x_2^* \, \rE_{\mathcal B}(x_1^*  w_k  y_1) \, y_2 \cdots y_{2n + 1})    \\
    & = 
 \rE_{\mathcal M_1}( x_{2m + 1}^* \cdots x_3^* \, \rE_{\mathcal B}(x_2^* \, \rE_{\mathcal B}(x_1^*  w_k  y_1) \, y_2) \, y_3 \cdots y_{2n + 1})      
    \eqstop
  \end{align*}
  
Take $\varepsilon > 0$. Since $\Tr(q) < + \infty$, we may choose a large enough finite trace projection $p \in \mathcal B$ such  that 
\begin{equation*}
  \|q y_1 - q y_1 p\|_2 + \|x_1^* q - p x_1^* q\|_2 \leq \varepsilon
  \eqstop
\end{equation*}
 We infer that $\|\rE_{\mathcal B}(x_1^* w_k y_1) - \rE_{\mathcal B}(px_1^*q \, w_k \, q y_1 p)\|_2 \leq \varepsilon$ for all $k$ and hence 
 \begin{equation*}
   \limsup_k \left\|\rE_{\mathcal M_1}(x^* w_k y) -  \rE_{\mathcal M_1}( x_{2m + 1}^* \cdots x_3^* \, \rE_{\mathcal B}(x_2^* \, \rE_{\mathcal B}(px_1^*q  \, w_k \,  qy_1p) \, y_2) \, y_3 \cdots y_{2n + 1})  \right\|_2 \leq \varepsilon
   \eqstop
 \end{equation*}
  Since the inclusion $\mathcal B \subset \mathcal M_2$ is mixing, since $(\rE_{\mathcal B}(px_1^*q \, w_k \, qy_1p))_k$ is a net in $\Ball(p \mathcal B p)$ that goes to~$0$ weakly and since $px_2, py_2 \in p(\mathcal M_2 \ominus \mathcal B)$, it follows that $\lim_k \|\rE_{\mathcal B}(x_2^*p \, \rE_{\mathcal B}(px_1^*q \,   w_k \, q y_1 p) \, py_2)\|_2 =~0$ and hence 
  \begin{equation*}
    \lim_k \left \| \rE_{\mathcal M_1}( x_{2m + 1}^* \cdots x_3^* \, \rE_{\mathcal B}(x_2^* \, \rE_{\mathcal B}(px_1^*q \,  w_k \, q  y_1p) \, y_2) \, y_3 \cdots y_{2n + 1}) \right \|_2  = 0
    \eqstop
  \end{equation*}
  This implies that $\limsup_k \|\rE_{\mathcal M_1}(x^* w_k y)\|_2 \leq \varepsilon$. Since $\varepsilon > 0$ is arbitrary, we deduce that $\lim_k \|\rE_{\mathcal M_1}(x^* w_k y)\|_2 = 0$. 
 \end{proof}

\begin{proof}[Proof of Theorem~\ref{thm:core-omega-solid}]
  By Proposition~\ref{prop:amplifications-and-omega-solidity}, it suffices to prove that finite corners of continuous cores of free Araki-Woods factors are $\omega$-solid.

  Let $U: \RR \ra \cO(H_\RR)$ be any orthogonal representation on a separable real Hilbert space that is the direct sum of a mixing representation and a representation of dimension less than or equal to~$1$.  Denote by $(M, \varphi) = (\Gamma(H_\RR, U_t)\dpr, \varphi_U)$ the associated free Araki-Woods factor together with its free quasi-free state and $\cM = \core_\vphi(M)$ its continuous core with respect to the free quasi-free state~$\varphi$. Observe that $M$ is a type ${\rm III_1}$ factor and hence $\mathcal M$ is a type ${\rm II_\infty}$ factor. Let $p \in \rL_\vphi(\RR)$ be any nonzero finite trace projection and $\mathcal Q \subset p \cM p$ any von Neumann subalgebra such that $\mathcal Q' \cap (p\cM p)^\omega$ is diffuse. 

  Assume by contradiction that $\mathcal Q$ is not amenable. Let $z \in \mathcal Z(\mathcal Q)$ be a nonzero central projection such that $\mathcal Q z$ has no amenable direct summand. Since $p \mathcal M p$ is a ${\rm II_1}$ factor and since $\rL_\varphi(\RR) p$ is diffuse, there exists $u \in \mathcal U(p \mathcal M p)$ and $q \in \rL_\varphi(\RR)p$ such that $u z u^* = q$. So up to conjugating by a unitary and taking a smaller projection in $\rL_\varphi(\RR)p$, we may assume without loss of generality that $\mathcal Q \subset p \mathcal M p$ has no amenable direct summand and that $\mathcal Q' \cap (p \mathcal M p)^\omega$ is diffuse.

  By Theorem~\ref{thm:spectral-gap}, we know that there exists a nonzero finite trace projection $q \in \rL_\varphi(\RR)$ such that $\mathcal Q' \cap (p \cM p) ^\omega  \preceq_{\cM ^\omega } (\rL_\vphi(\RR)q)^\omega$. Up to replacing $q$ by $p \vee q \in \rL_\varphi(\RR)$, we may assume that $p \leq q$. 
  
If $(U_t)$ is mixing, then \cite[Proposition 2.4]{houdayer10} and Proposition \ref{prop:mixing-modular} show that the inclusion $\rL_\varphi(\RR)\subset \cM $ is mixing. Applying \cite[Lemma 9.5]{ioana12}, we obtain that $\mathcal Q \preceq_{q \cM q} \rL_\varphi(\RR)q$.

 If $(U_t)$ is the direct sum of a mixing orthogonal representation with an orthogonal representation of dimension one, then $M = N \ast \rL(\ZZ)$, where $N$ is the Araki-Woods factor associated with the mixing part of $(U_t)$.  Writing $\cN = \core_{\varphi_{ | \mathcal N}}(N)$, we obtain $\cM \cong \cN *_{\rL_\varphi(\RR)} (\rL(\ZZ) \vnt \rL_\varphi(\RR))$.  Hence Proposition~\ref{prop:mixing-afp} shows that the inclusion $\rL(\ZZ) \vnt \rL_\varphi(\RR) \subset  \cM $ is mixing. Moreover, we know that $\mathcal Q' \cap (p \cM p)^\omega  \preceq_{(q \cM q)^\omega}( \rL_\varphi(\RR)q)^\omega$ and hence $Q' \cap (p \cM p)^\omega  \preceq_{(q \cM q)^\omega} ((\rL(\ZZ) \vnt \rL_\varphi(\RR))q)^\omega$.  Applying \cite[Lemma 9.5]{ioana12}, we obtain that $\mathcal Q \preceq_{q \cM q} (\rL(\ZZ) \vnt \rL_\varphi(\RR))q$.  
 
 However, in both cases, this contradicts the fact that $\mathcal Q$ has no amenable direct summand.
\end{proof}

\subsection{Computation of Connes's $\tau$-invariant for $\omega$-solid factors}

Let $M$ be any von Neumann algebra with separable predual. We endow $\Aut(M)$ with the topology of pointwise convergence in $M_\ast$, that is, 
\begin{equation*}
  \alpha_n \to \id \text{ in } \Aut(M) \text{ as } n \to \infty \; \text{ if and only if } \; \lim_{n \to \infty} \|\varphi \circ \alpha_n - \varphi\| = 0 \text{ for all } \varphi \in M_\ast  
  \eqstop
\end{equation*}
Endowed with this topology, $\Aut(M)$ becomes a Polish group. 

Recall from \cite{connes74-almost-periodic} that when $M$ is a factor, we have that $M$ is {\em full} if and only if the subgroup $\Inn(M)$ of inner automorphisms is closed in $\Aut(M)$. Equivalently, we have $M_\omega = \CC 1$ for some (or any) $\omega \in \beta(\NN) \setminus \NN$. In that case, the quotient group $\Out(M) = \Aut(M) /\Inn(M)$ endowed with the quotient topology is a Polish group. We will denote by ${\pi : \Aut(M) \to \Out(M)}$ the quotient homomorphism.

By Connes's Radon-Nikodym cocycle theorem \cite[Th\'eor\`eme 1.2.1]{connes73-type-III} (see also \cite[Theorem VIII.3.3]{takesaki03-III}), the homomorphism $\delta : \RR \to \Out(M) : t \mapsto \pi(\sigma_t^\varphi)$ is well-defined and does not depend on the choice of a particular state on $M$.

\begin{definition}[\cite{connes74-almost-periodic}]
Let $M$ be a full factor of type ${\rm III_1}$ with separable predual. We define $\tau(M)$ to be the weakest topology that makes the map $\delta : \RR \to \Out(M)$ continuous.
\end{definition}

It is typically difficult to calculate Connes's $\tau$-invariant for arbitrary type ${\rm III_1}$ factors. In the case of the free Araki-Woods factors $M = \Gamma(H_{\RR}, U_t)\dpr$, using a $14 \varepsilon$-type argument, it is proven in \cite{shlyakhtenko97,vaes04-etas-quasi-libres-libres} that $\tau(M)$ is the weakest topology that makes the map $\RR \to \mathcal O(H_\RR) : t \mapsto U_t$ strongly continuous.

In the next proposition, we show that Connes's $\tau$-invariant is computable for a fairly large class of $\omega$-solid type ${\rm III_1}$ factors. Our proof no longer relies on a $14 \varepsilon$-type argument and works in great generality.

\begin{proposition}
  \label{prop:calculation-tau}
  Let $M$ be any $\omega$-solid factor of type ${\rm III_1}$ with separable predual and $\vphi \in M_*$ any faithful normal state whose centralizer is a non-amenable ${\rm II_1}$ factor.  Then $M$ is a full factor and $\tau(M)$ is the weakest topology on $\RR$ that makes the map $\RR \to \Aut(M) : t \mapsto \sigma_t^\vphi$ continuous.
\end{proposition}

\begin{proof}
  Let $M$ be any $\omega$-solid factor of type ${\rm III_1}$ with separable predual and $\vphi \in M_*$ a faithful normal state whose centralizer is a non-amenable ${\rm II_1}$ factor.  Since $M$ is a non-amenable $\omega$-solid factor, $M' \cap M^\omega$ is discrete by Proposition \ref{prop:omega-solid-characterisation} and hence $M' \cap M^\omega = \CC 1$ by Corollary \ref{cor:asymptotic-centraliser-diffuse}. This implies that $M$ is a full factor. We next have to show that if $(t_n)_n$ is a sequence in $\RR$ that converges to $0$ with respect to $\tau(M)$, then $\sigma_{t_n}^\vphi \ra \id$ in $\Aut(M)$.

By Theorem \ref{thm:omega-solidity} and Proposition \ref{prop:omega-solid-characterisation}, the relative commutant $(M^\varphi)' \cap M^\omega$ is discrete. Applying Theorem \ref{thm:diffuse-centraliser}, we have that $(M^\varphi)' \cap M^\omega = (M^\varphi)' \cap M$. Since $((M^\varphi)' \cap M)^\varphi = (M^\varphi)' \cap M^\varphi = \CC 1$, \cite[Lemma 5.3]{andohaagerup12} implies that $(M^\varphi)' \cap M = \CC1$ or $(M^\varphi)' \cap M$ is a factor of type ${\rm III}_1$. Since $(M^\varphi)' \cap M^\omega = (M^\varphi)' \cap M$ is discrete, we obtain that $(M^\varphi)' \cap M^\omega = \CC1$. Observe that this implies that $(M^\varphi)' \cap M^\omega = \CC1$ for {\em all} non-principal ultrafilter $\omega \in \beta(\NN) \setminus \NN$.

Now take a sequence $(t_n)_n$ in $\RR$ that converges to $0$ with respect to $\tau(M)$.  Then there is a sequence of unitaries $(u_n)_n$ in $M$ such that $(\Ad u_n) \circ \sigma_{t_n}^\vphi \ra \id$ in $\Aut(M)$.  Fix $\omega \in \beta(\NN) \setminus \NN$ a non-principal ultrafilter. As in the proof of \cite[Proposition 3.1]{ueda11-free-products}, we have  that $(u_n)_n \in \mathcal M^\omega(M)$ and $(u_n)^\omega \in (M^\varphi)' \cap M^\omega$.  Indeed, for all $n \in \NN$, we have
  \begin{equation*}
   \|u_n^* \varphi - \varphi u_n^*\| =  \| \vphi \circ (\Ad u_n) - \vphi\|
    =
    \| \vphi \circ (\Ad u_n) \circ \sigma_{t_n}^\vphi - \vphi \circ \sigma_{t_n}^\vphi\|
    =
    \| \vphi \circ (\Ad u_n) \circ \sigma_{t_n}^\vphi - \vphi\|.
    \end{equation*}
  Since $\lim_{n \to \infty}  \| \vphi \circ (\Ad u_n) \circ \sigma_{t_n}^\vphi - \vphi\| = 0$, we have $\lim_{n \to \omega}  \| \vphi \circ (\Ad u_n) \circ \sigma_{t_n}^\vphi - \vphi\| = 0$ and hence $\lim_{n \to \omega} \|u_n^* \varphi - \varphi u_n^*\| =0$. Therefore $(u_n)_n \in \mathcal M^\omega(M)$ and $(u_n)^\omega \in (M^\omega)^{\varphi^\omega}$ by Proposition \ref{prop:representation-by-projections}. We moreover have $(\Ad u_n) \circ \sigma_{t_n}^\vphi(x) \to x$ strongly as $n \to \infty$ for all $x \in M^\varphi$. This implies that $\lim_{n \to \omega} \|u_n x u_n^* - x\|_\varphi = 0$ for all $x \in M^\varphi$. Since $(u_n)_n \in \mathcal M^\omega(M)$ and $(u_n)^\omega \in (M^\omega)^{\varphi^\omega}$, we finally obtain $(u_n)^\omega \in (M^\varphi)' \cap M^\omega$.

Since $(M^\varphi)' \cap M^\omega = \CC 1$, we have $\lim_{n \to \omega} \|u_n - \varphi(u_n) 1\|_\varphi = \|(u_n)^\omega - \varphi^\omega((u_n)^\omega)\|_{\varphi^\omega} = 0$. Since this is true for every $\omega \in \beta(\NN) \setminus \NN$, we obtain $\lim_{n \to \infty} \|u_n - \varphi(u_n) 1\|_\varphi = 0$.
 
Proceeding now exactly as in the proof of \cite[Theorem 5.2]{connes74-almost-periodic}, we conclude that $\sigma_{t_n}^\vphi \ra \id$ in $\Aut(M)$.
\end{proof}

\section{Proof of Theorem \ref{thm:dichotomy-FAW}}
\label{sec:dichotomy}

We first recall a basic fact on $\varepsilon$-orthogonality.

\begin{definition}
  \label{def:epsilon-orthogonality}
  Let $H$ be a complex Hilbert space and $\veps \geq 0$.  Two (not necessarily closed) subspaces $K, L \subset H$ are called $\veps$-{\em orthogonal} if $|\langle \xi, \eta \rangle| \leq \veps \|\xi\| \|\eta\|$ for all $\xi \in K$ and all $\eta \in L$. In that case, we will denote $K \perp_\varepsilon L$.
\end{definition}

\begin{proposition}[{\cite[Proposition 2.3]{houdayer12}}]
  \label{prop:epsiolon-orthogonality}
  There is a continuous function $\delta: [0, 1/2) \ra \RR_{\geq 0}$ satisfying $\delta(0) = 0$ and the following property.  If $k \geq 1$ and $0 \leq \veps < 1/2$ are such that $\delta^{\circ(k-1)}(\veps) < 1/2$, then for all projections $p_i \in \bo(H)$, $i \in \{1, \dots, 2^k\}$, satisfying $p_i H \perp_\veps p_j H$ for all $i, j \in \{1, \dotsc, 2^k\}$, $i \neq j$, we have
  \begin{equation*}
    \sum_{i = 1}^{2^k} \|p_i \xi\|^2
    \leq
    \prod_{j = 0}^{k -1 }(1 + \delta^{\circ j}(\veps))^2 \|P \xi\|^2
    \eqcomma
  \end{equation*}
  where $P = \bigvee_{i = 1}^{2^k}p_i$ is the projection onto the closed linear span $\cspan \bigcup_{i = 1}^{2^k} p_i H$.
\end{proposition}

The main result of this section is the following {\em asymptotic orthogonality} result in the framework of ultraproducts of free Araki-Woods factors and is inspired by \cite[Lemma 2.1]{popa83-maximal-injective-factors}.

\begin{theorem}\label{thm:asymptotic-orthogonality}
  Let $U : \RR \ra \cO(H_\RR)$ be any weakly mixing orthogonal representation on a separable real Hilbert space and $(M, \varphi) = (\Gamma(H_\RR, U_t)\dpr, \varphi_U)$ the associated free Araki-Woods factor. Then for all $x , y \in (M^\omega)^{\vphi^\omega} \ominus \CC 1$ and all $a, b \in M \ominus \CC 1$, we have $\vphi^\omega(b^* y^* a x) = 0$.
\end{theorem}
\begin{proof}
  Let $H = H_\RR \oplus  {\rm i} H_\RR$ and denote by $\cH = \cF(H)$ the full Fock space. We view $K_\RR + {\rm i} K_\RR \subset H$ as a dense subspace of $H$. Put $\kappa_t = \id \oplus \bigoplus_{n \geq 1} U_t^{\ot n} \in \cU(\cH)$. For every $x \in M$, we have 
  \begin{equation*}
    \sigma_t^\vphi(x) \Omega
    =
    \kappa_t (x \Omega)
    \eqstop
  \end{equation*}
  Since the linear span of $1$ and of all the reduced words $W(\xi_1\ot \cdots \ot \xi_m)$ with $m \geq 1$ and $\xi_j \in K_\RR + {\rm i} K_\RR$ is a unital $\sigma$-strongly dense $\ast$-subalgebra of $M$, it suffices to prove the result when $a = W(\xi_1 \ot \cdots \ot \xi_k)$ and $b = W(\eta_1 \ot \cdots \ot \eta_\ell)$ are reduced words with $\xi_1, \dots , \xi_k, \eta_1, \dots , \eta_\ell \in K_\RR + {\rm i} K_\RR$. Approximating $\ol\eta_j \in K_\RR + {\rm i} K_\RR$ by $\mathbf 1_{[\lambda^{-1}, \lambda]}(A)(\ol \eta_j) \in K_\RR +{\rm i} K_\RR$ for all $1 \leq j \leq \ell$ and for $\lambda > 1$ sufficiently large, we may further assume that $\ol\eta_j = \mathbf 1_{[\lambda^{-1}, \lambda]}(A)(\ol\eta_j)$ for all $1 \leq j \leq \ell$. It follows that the map $\RR \to K_\RR + {\rm i} K_\RR : t \mapsto U_t \ol\eta_j$ can be extended to an entire analytic function which takes values in $K_\RR + {\rm i} K_\RR$ for all $1 \leq j\leq \ell$. This implies that the map $\RR \to M : t \mapsto W(U_t \ol\eta_\ell \otimes \cdots \otimes U_t \ol\eta_1)$ can be extended to an $M$-valued entire analytic function. Since $\sigma_t^\varphi(W(\ol\eta_\ell \otimes \cdots \otimes \ol \eta_1)) = W(U_t \ol\eta_\ell \otimes \cdots \otimes U_t \ol\eta_1)$ for all $t \in \RR$, we obtain that $W(\ol\eta_\ell \otimes \cdots \otimes \ol\eta_1)$ is analytic for the modular automorphism group $(\sigma_t^\varphi)$ and we have $\sigma_z^\varphi(W(\ol\eta_\ell \otimes \cdots \otimes \ol\eta_1)) = W(A^{{\rm i}z}\ol\eta_\ell \otimes \cdots \otimes A^{{\rm i} z}\ol\eta_1)$ for all $z \in \CC$.

  From now on and for the rest of the proof, define $L = \lspan(\xi_k, \ol \xi_k, \eta_1, \ol \eta_1) \subset K_\RR + {\rm i} K_\RR$.  We will use the following notation:
  \begin{itemize}
  \item $\cX_1 \subset \cH$ is the closed subspace generated by the linear span of all the reduced words $e_1 \ot \cdots \ot e_n$ with $n \geq 1$ and such that $e_1 \in L$.

  \item $\cX_2 \subset \cH$ is the closed subspace generated by the linear span of all the reduced words $e_1 \ot \cdots \ot e_n$ with $n \geq 1$ and such that $\ol e_n \in L$.

  \item $\cY \subset \cH$ is the closed subspace generated by the linear span of all the reduced words $e_1 \ot \cdots \ot e_n$ with $n \geq 1$ and such that $e_1, \ol e_n \in (K_\RR + {\rm i} K_\RR) \cap L^\perp$.
  \end{itemize}
  Observe that we have
  \begin{equation*}
    \CC \Omega \oplus \ol{(\cX_1 + \cX_2)}^{\|\cdot\|_\vphi} \oplus \cY
    =
    \cH
    \eqstop
  \end{equation*}

  \textbf{Claim 1.}
  Let $\varepsilon \geq 0$ and $t \in \RR$ be such that $U_t(L) \perp_{\veps / \dim(L)} L$. Then for all $i \in \{1, 2\}$, we have 
  \begin{equation*}
    \kappa_t (\cX_i) \perp_\veps \cX_i
    \eqstop
  \end{equation*}

 Choose an orthonormal basis $(\zeta_1, \dots , \zeta_{\dim(L)})$ of $L$.  We first prove the claim for $\cX_1$. We will identify $\cX_1$ with $L \ot \cH$ using the following unitary defined by
  \begin{equation*}
    \cV_1 :
    H \ot \cH \ni
    \zeta \ot (e_1 \ot \cdots \ot e_n)
    \mapsto
    \zeta \ot e_1 \ot \cdots \ot e_n
    \in \mathcal H
    \eqcomma
  \end{equation*}
  for all $n \geq 1$ and all $\zeta, e_1, \dots, e_n \in H$.  Observe that $\kappa_t \cV_1 = \cV_1 (U_t \ot \kappa_t)$ for all $t \in \RR$. Let $\xi, \eta \in \cX_1$ be such that $\xi = \sum_{i = 1}^{\dim(L)} \zeta_i \ot \mu_i$ and $\eta = \sum_{j = 1}^{\dim(L)} \zeta_j \ot \nu_j$ with $\mu_i, \nu_j \in \cH$. Further observe that $\|\xi\|^2 = \sum_{i = 1}^{\dim(L)} \|\mu_i\|^2$ and $\|\eta\|^2 = \sum_{j = 1}^{\dim(L)} \|\nu_j\|^2$. We have $\kappa_t \xi = \sum_{i = 1}^{\dim(L)} U_t \zeta_i \ot \kappa_t \mu_i$ and hence
  \begin{equation*}
    |\langle \kappa_t \xi, \eta\rangle |
    \leq
    \sum_{i, j = 1}^{\dim(L)} |\langle U_t \zeta_i, \zeta_j\rangle| \|\mu_i\| \|\nu_j\|
    \eqstop
  \end{equation*}
    Since $|\langle U_t \zeta_i, \zeta_j\rangle| \leq \veps / \dim(L)$, we obtain $|\langle \kappa_t \xi, \eta\rangle | \leq \veps \|\xi\| \|\eta\|$  by the Cauchy-Schwarz inequality.

  Next, we prove the claim for $\cX_2$.  We identify $\cX_2$ with $\cH \ot L$ using the unitary  defined by
  \begin{equation*}
    \cV_2 :
    \cH \ot H \ni
    (e_1 \ot \cdots \ot e_n) \ot \zeta
    \mapsto
    e_1 \ot \cdots \ot e_n \ot \zeta
    \in \mathcal H
    \eqcomma
  \end{equation*}
  for all $n \geq 1$ and all $\zeta, e_1, \dots, e_n \in H$. Observe that $\kappa_t \cV_2 = \cV_2 ( \kappa_t \ot U_t)$ for all $t \in \RR$. Let $\xi, \eta \in \cX_2$ be such that $\xi = \sum_{i = 1}^{\dim(L)} \mu_i \ot \zeta_i$ and $\eta = \sum_{j = 1}^{\dim(L)}  \nu_j \ot \zeta_j$ with $\mu_i, \nu_j \in \cH$. Further observe that $\|\xi\|^2 = \sum_{i = 1}^{\dim(L)} \|\mu_i\|^2$ and $\|\eta\|^2 = \sum_{j = 1}^{\dim(L)} \|\nu_j\|^2$.  We have $\kappa_t \xi = \sum_{i = 1}^{\dim(L)} \kappa_t \mu_i \ot U_t \zeta_i$ and hence
  \begin{equation*}
    |\langle \kappa_t \xi, \eta\rangle |
    \leq
    \sum_{i, j = 1}^{\dim(L)} |\langle U_t \zeta_i, \zeta_j\rangle| \|\mu_i\| \|\nu_j\|
    \eqstop
  \end{equation*}
  Since $|\langle U_t \zeta_i, \zeta_j\rangle| \leq \veps / \dim(L)$, we obtain $|\langle \kappa_t \xi, \eta\rangle | \leq \veps \|\xi\| \|\eta\|$  by the Cauchy-Schwarz inequality. This finishes the proof of the claim.

\textbf{Claim 2.}
For every $x = (x_n)^\omega \in (M^\omega)^{\vphi^\omega}$, we have 
\begin{equation*}
  \lim_{n \to \omega} \|P_{\cX_1}(x_n \Omega)\|_\vphi
  =
  0 
  \quad \text{ and } \quad
  \lim_{n \to \omega} \|P_{\cX_2}(x_n \Omega)\|_\vphi
  =
  0
  \eqstop
\end{equation*}

Let $x \in (M^\omega)^{\vphi^\omega}$. We may assume that $x \in \Ball((M^\omega)^{\vphi^\omega})$ and then choose a sequence $(x_n)_n \in \mathcal M^\omega(M)$ such that $x_n \in \Ball(M)$ for all $n \in \NN$ and $x = (x_n)^\omega$. For all $i \in \{1, 2\}$, all $t \in \RR$ and all $n \in \NN$, we have
\begin{align*}
  \| P_{\cX_i}(x_n \Omega) \|_\vphi^2
  & = \|\kappa_t P_{\cX_i} (x_n \Omega) \|_\vphi^2 \\
  & \leq 2 \|\kappa_t P_{\cX_i} (x_n \Omega) - P_{\kappa_t (\cX_i)} (x_n \Omega)\|_\vphi^2 + 2 \| P_{\kappa_t ( \cX_i)} (x_n \Omega)\|_\vphi^2 \\
  & = 2 \| P_{\kappa_t ( \cX_i)} (\kappa_t(x_n \Omega) - x_n \Omega)\|_\vphi^2 + 2 \| P_{\kappa_t ( \cX_i)} (x_n \Omega)\|_\vphi^2 \\
  & \leq 2 \|\sigma_t^\vphi(x_n) - x_n\|_\vphi^2 + 2 \| P_{\kappa_t  (\cX_i)} (x_n \Omega)\|_\vphi^2.
\end{align*}
Furthermore, \cite[Theorem~4.1]{andohaagerup12} says that for all $t \in \RR$
\begin{equation*}
  (x_n)^\omega
  =
  x
  =
  \sigma_t^{\vphi^\omega}(x) = (\sigma_t^\vphi(x_n))^\omega
\end{equation*}
holds.  This implies that $\lim_{n \to \omega} \|x_n - \sigma_t^{\vphi}(x_n)\|_\vphi^\# = 0$ for all $t \in \RR$.

Fix $p \geq 1$.  Choose $\veps > 0$ very small according to Proposition~\ref{prop:epsiolon-orthogonality} so that $\prod_{j = 0}^{p - 1}(1 + \delta^{\circ j}(\veps))^2 \leq 2$.  Since $U : \RR \to \mathcal O(H_\RR)$ is weakly mixing and since $L$ is finite dimensional, with $\veps' = \veps/ \dim(L)$, we can choose inductively $t_1, \dots, t_{2^p} \in \RR$ such that 
\begin{equation*}
  U_{t_j} (L) \perp_{\veps'} U_{t_i} (L), \forall 1 \leq i < j \leq 2^p
  \eqstop
\end{equation*}
Using Claim 1, this implies that
\begin{equation*}
  \kappa_{t_j} (\cX_1) \perp_{\veps} \kappa_{t_i} (\cX_1)
  \text{ and }
  \kappa_{t_j} (\cX_2) \perp_{\veps} \kappa_{t_i} (\cX_2), \forall 1 \leq i < j \leq 2^p
  \eqstop
\end{equation*}
Thus, using the above inequalities and Proposition \ref{prop:epsiolon-orthogonality}, we obtain
\begin{align*}
  \lim_{n \ra \omega} 2^p \| P_{\cX_i}(x_n \Omega) \|_\vphi^2
 &=  \lim_{n \ra \omega} \sum_{j = 1}^{2^p} \|\kappa_{t_j} P_{\cX_i} (x_n \Omega) \|_\vphi^2  \\
&  \leq
\lim_{n \to \omega} \sum_{j = 1}^{2^p} 2 \|\sigma_{t_j}^\vphi(x_n) - x_n\|_\vphi^2 + \lim_{n \ra \omega} \sum_{j = 1}^{2^p}  2 \|P_{\kappa_{t_j}(\cX_i)}(x_n \Omega)\|_\vphi^2 \\
 & \leq
  \lim_{n \ra \omega} 4 \|x_n\|_\vphi^2
  \eqstop
\end{align*}
We conclude that $\lim_{n \to \omega} \| P_{\cX_i}(x_n \Omega) \|_\vphi^2 \leq 2^{2 - p}$ for all $p \geq 1$. Thus, we have $\lim_{n \to \omega} \| P_{\cX_i}(x_n \Omega) \|_\vphi = 0$. This finishes the proof of the claim.

\textbf{Claim 3.}
The subspaces $W(\xi_1 \ot \cdots \ot \xi_k) \, \cY$ and $J_\vphi \sigma_{- {\rm i}/2}^\varphi(W( \ol \eta_\ell \ot \cdots \ot  \ol \eta_1)) J_\vphi \, \cY$ are orthogonal in $\cH$. \\
Let $m, n \geq 1$ and $e_1, \dots, e_m, f_1 \dots, f_n \in K_\RR + {\rm i} K_\RR$. Assume moreover that $e_1, \ol e_m, f_1, \ol f_n \in L^\perp$ so that $e_1 \ot \cdots \ot e_m \in \cY$ and $f_1 \ot \cdots \ot f_n \in \cY$. Then by Proposition~\ref{prop:wick} (ii) and since $\xi_k, \eta_1 \in L$, we have
\begin{align*}
  W(\xi_1 \ot \cdots \ot \xi_k) \, (e_1 \ot \cdots \ot e_m)
  & = W(\xi_1 \ot \cdots \ot \xi_k)  W(e_1 \ot \cdots \ot e_m) \Omega \\
  & = W(\xi_1 \ot \cdots \ot \xi_k \ot e_1 \ot \cdots \ot e_m) \Omega \\
  & = \xi_1 \ot \cdots \ot \xi_k \ot e_1 \ot \cdots \ot e_m \eqcomma \\
  J_\vphi \sigma_{-{\rm i}/2}^\varphi( W( \ol \eta_\ell \ot \cdots \ot \ol \eta_1) )J_\vphi \, (f_1 \ot \cdots \ot f_n) &= W(f_1 \ot \cdots \ot f_n) W(\eta_1 \ot \cdots \ot \eta_\ell) \Omega \\
  &= W(f_1 \ot \cdots \ot f_n \ot \eta_1 \ot \cdots \ot \eta_\ell) \Omega \\
  &= f_1 \ot \cdots \ot f_n \ot \eta_1 \ot \cdots \ot \eta_\ell.
\end{align*}
Since $\langle \xi_1, f_1 \rangle = 0$, we see that the vectors
\begin{equation*}
  W(\xi_1 \ot \cdots \ot \xi_k) \, (e_1 \ot \cdots \ot e_m)
  \quad \text{ and } \quad
  J_\vphi \sigma_{-{\rm i}/2}^\varphi( W( \ol \eta_\ell \ot \cdots \ot \ol \eta_1) )J_\vphi \, (f_1 \ot \cdots \ot f_n)
\end{equation*}
are orthogonal in $\cH$. Finally, using the density of the linear span of the words $e_1 \ot \cdots \ot e_m$ and $f_1 \ot \cdots \ot f_n$ in $\cY$ finishes the proof of the claim.

We are now ready to finish the proof of Theorem~\ref{thm:asymptotic-orthogonality}. Let $x, y \in (M^\omega)^{\vphi^\omega} \ominus \CC 1$. Using Claim 2 and the fact that $\lim_{n \to \omega} \|P_{\CC \Omega}(x_n \Omega)\|_\vphi = 0$, we have
\begin{align*}
  \Lambda_{\vphi^\omega}(ax)
  & =
  (W(\xi_1 \ot \cdots \ot  \xi_k) \, x_n \Omega)_\omega \\
  & =
  (W(\xi_1 \ot \cdots \ot \xi_k) \, P_{\cY}(x_n \Omega))_\omega \eqcomma \\
  \Lambda_{\vphi^\omega}(yb)
  & =
  (J_\vphi \sigma_{-{\rm i}/2}^\varphi( W( \ol \eta_\ell \ot \cdots \ot \ol \eta_1) )J_\vphi \, y_n \Omega)_\omega \\
  & =
  (J_\vphi \sigma_{-{\rm i}/2}^\varphi( W( \ol \eta_\ell \ot \cdots \ot \ol \eta_1) )J_\vphi \, P_{\cY}(y_n \Omega))_\omega
  \eqstop
\end{align*}
By Claim 3, we know that
\begin{equation*}
  W(\xi_1 \ot \cdots \ot \xi_k) \, P_{\cY}(x_n \Omega)
  \perp
  J_\vphi \sigma_{-{\rm i}/2}^\varphi( W( \ol \eta_\ell \ot \cdots \ot \ol \eta_1) )J_\vphi \, P_{\cY}(y_n \Omega)
  \eqcomma
\end{equation*}
for all $n \in \NN$.  Hence $\Lambda_{\vphi^\omega}(ax) \perp \Lambda_{\vphi^\omega}(yb)$ in $\cH^\omega$, which implies that $\vphi^\omega(b^*y^* a x) = 0$.
\end{proof}

\begin{theorem}\label{thm:subalgebra-commutant}
  Let $U : \RR \ra \cO(H_\RR)$ be any weakly mixing orthogonal representation on a separable real Hilbert space and $(M, \varphi) = (\Gamma(H_\RR, U_t)\dpr, \varphi_U)$ the associated free Araki-Woods factor. Let $Q \subset M$ be any von Neumann subalgebra such that $Q' \cap (M^\omega)^{\vphi^\omega} \neq \CC 1$. Then  $Q = \CC 1$.
\end{theorem}
\begin{proof}
  Assume that $Q' \cap (M^\omega)^{\vphi^\omega} \neq \CC 1$. Choose a projection $e \in Q' \cap (M^\omega)^{\vphi^\omega}$ such that $e \notin \{0, 1\}$. Then choose a sequence of projections $(e_n)_n \in \mathcal M^\omega(M)$ such that $e = (e_n)^\omega$ and $\lim_{n \to \omega} \|\sigma_t^\vphi(e_n) - e_n\|_\vphi^\# = 0$ for all $t \in \RR$. Put $a = \sigma\text{-weak} \lim_{n \to \omega} e_n \in Q' \cap M^\varphi$. Since $M^\vphi = \CC 1$, we obtain $a = \varphi(a) 1$. Since $e \notin \{ 0, 1 \}$, we have $\varphi(a) \notin \{ 0, 1\}$.

  Let $y \in Q \ominus \CC 1$. By Theorem~\ref{thm:asymptotic-orthogonality}, we have
  \begin{equation*}
   \|(e - \varphi(a) 1) y\|_{\vphi^\omega}^2
    =
    \varphi^\omega(y^* (e - \varphi(a) 1)^* (e - \varphi(a) 1) y)
    =
   \varphi^\omega(y^* (e - \varphi(a) 1)^* y (e - \varphi(a) 1) )
    =
    0
  \end{equation*}
  Moreover,
  \begin{equation*}
    \|(e - \varphi(a) 1) y\|_{\vphi^\omega}^2
    =
    \lim_{n \to \omega} \|(e_n - \varphi(a) 1) y\|_\vphi^2
    =
    \lim_{n \to \omega} \langle (e_n - 2 \varphi(a) e_n + \varphi(a)^2 1) y \Omega , y \Omega \rangle_\varphi
    =
    (\varphi(a) - \varphi(a)^2)\|y\|_\vphi^2
    \eqstop
  \end{equation*}
  Since $\varphi(a) \notin \{ 0, 1\}$, it follows that $y = 0$ and hence $Q = \CC 1$.
\end{proof}

\begin{proof}[Proof of Theorem~\ref{thm:dichotomy-FAW}]
  Let $Q \subset M$ be any von Neumann subalgebra that is globally invariant under the modular automorphism group $(\sigma_t^\vphi)$. There is a unique $\varphi$-preserving faithful normal conditional expectation $\rE_Q : M \to Q$. Assume that $Q' \cap M^\omega \neq \CC 1$. Then we have $Q' \cap (M^\omega)^{\varphi^\omega} \neq \CC 1$ by Lemma \ref{lem:diffuse-centraliser}. Therefore, we obtain $Q = \CC 1$ by Theorem~\ref{thm:subalgebra-commutant}.
\end{proof}

\bibliographystyle{mybibtexstyle}
\bibliography{operatoralgebras}

{\small \parbox[t]{200pt}{Cyril Houdayer\\CNRS - Universit\'e Paris-Est - Marne-la-Vall\'ee \\
    LAMA UMR 8050 \\ 77454 Marne-la-Vall\'ee cedex~2 
\\ France
    \\ {\footnotesize cyril.houdayer@u-pem.fr}}
  \hspace{15pt}
  \parbox[t]{200pt}{Sven Raum\\
RIMS\\
Kitashirakawa-oiwakecho\\
606-8502 Sakyo-ku, Kyoto\\
Japan
    \\ {\footnotesize sven.raum@gmail.com}}}

\end{document}

%% file: shortcuts.tex
\newcommand{\cF}{\ensuremath{\mathcal{F}}}

\newcommand{\cH}{\ensuremath{\mathcal{H}}}

\newcommand{\cM}{\ensuremath{\mathcal{M}}}
\newcommand{\cN}{\ensuremath{\mathcal{N}}}
\newcommand{\cO}{\ensuremath{\mathcal{O}}}

\newcommand{\cQ}{\ensuremath{\mathcal{Q}}}

\newcommand{\cU}{\ensuremath{\mathcal{U}}}
\newcommand{\cV}{\ensuremath{\mathcal{V}}}

\newcommand{\cX}{\ensuremath{\mathcal{X}}}
\newcommand{\cY}{\ensuremath{\mathcal{Y}}}
\newcommand{\cZ}{\ensuremath{\mathcal{Z}}}

\newcommand{\rE}{\ensuremath{\mathrm{E}}}

\newcommand{\rL}{\ensuremath{\mathrm{L}}}

\newcommand{\veps}{\ensuremath{\varepsilon}}

\newcommand{\vphi}{\ensuremath{\varphi}}

\newcommand{\ol}{\overline}

\newcommand{\amid}{\ensuremath{\, | \,}}

\newcommand{\eqstop}{\ensuremath{\, \text{.}}}
\newcommand{\eqcomma}{\ensuremath{\, \text{,}}}

\newcommand{\NN}{\ensuremath{\mathbb{N}}}
\newcommand{\ZZ}{\ensuremath{\mathbb{Z}}}

\newcommand{\RR}{\ensuremath{\mathbb{R}}}
\newcommand{\CC}{\ensuremath{\mathbb{C}}}

\newcommand{\id}{\ensuremath{\mathrm{id}}}
\newcommand{\ra}{\ensuremath{\rightarrow}}

\newcommand{\Tr}{\ensuremath{\mathop{\mathrm{Tr}}}}

\newcommand{\Aut}{\ensuremath{\mathrm{Aut}}}

\newcommand{\ot}{\ensuremath{\otimes}}

\newcommand{\Cstar}{\ensuremath{\text{C}^*}}

\newcommand{\bo}{\ensuremath{\mathcal{B}}}
\newcommand{\ko}{\ensuremath{\mathcal{K}}}

\newcommand{\vnt}{\mathbin{\overline{\otimes}}}

\newcommand{\Linfty}{\ensuremath{{\offinterlineskip \mathrm{L} \hskip -0.3ex ^\infty}}}

\newcommand{\Lone}{\ensuremath{{\offinterlineskip \mathrm{L} \hskip -0.3ex ^1}}}
\newcommand{\ltwo}{\ensuremath{\ell^2}}

\newcommand{\linfty}{\ensuremath{{\offinterlineskip \ell \hskip 0ex ^\infty}}}

\newcommand{\lspan}{\ensuremath{\mathop{\mathrm{span}}}}
\newcommand{\cspan}{\ensuremath{\mathop{\overline{\mathrm{span}}}}}

\newcommand{\Inn}{\ensuremath{\mathrm{Inn}}}
\newcommand{\Ad}{\ensuremath{\mathop{\mathrm{Ad}}}}
\newcommand{\Out}{\ensuremath{\mathrm{Out}}}

